\documentclass[reqno]{amsart}

\usepackage[hmargin=35mm, vmargin=35mm]{geometry}
\usepackage[english]{babel}
\usepackage{amsmath,amssymb,amsthm,amscd}
\usepackage{parskip}
\usepackage{mathrsfs}
\usepackage{dsfont}
\usepackage{lmodern}
\usepackage{mathtools}
\usepackage{enumerate}
\usepackage{graphicx}
\usepackage{hyperref}
\usepackage[all,cmtip]{xy}
\usepackage{color}

\newtheorem{thmintro}{Theorem}

\newtheorem{corintro}[thmintro]{Corollary}
\newtheorem{propintro}[thmintro]{Proposition}
\newtheorem{theorem}{Theorem}[section]
\newtheorem{corollary}[theorem]{Corollary}
\newtheorem{lemma}[theorem]{Lemma}
\newtheorem{prop}[theorem]{Proposition}
\theoremstyle{definition}
\newtheorem{remark}[theorem]{Remark}
\newtheorem{remarkintro}{Remark}

\newtheorem{definition}[theorem]{Definition}

\newcommand{\CC}{\mathbb{C}}
\newcommand{\Ga}{\mathbb{G}_a}
\newcommand{\FF}{\mathbb{F}}
\newcommand{\KK}{\mathbb{K}}
\newcommand{\NN}{\mathbb{N}}

\newcommand{\ZZ}{\mathbb{Z}}

\newcommand{\CCC}{\mathcal{C}}
\newcommand{\DDD}{\mathcal{D}}
\newcommand{\WW}{\mathcal{W}}

\newcommand{\B}{\mathfrak{B}}

\newcommand{\G}{\mathfrak{G}}

\newcommand{\PP}{\mathfrak{P}}
\newcommand{\T}{\mathfrak{T}}
\newcommand{\U}{\mathfrak{U}}
\newcommand{\X}{\mathfrak{X}}

\newcommand{\oU}{\overline{U}}
\newcommand{\oB}{\overline{B}}
\newcommand{\oP}{\overline{P}}
\newcommand{\oG}{\overline{G}}
\newcommand{\opi}{\overline{\pi}}
\newcommand{\wpi}{\widetilde{\pi}}

\newcommand{\inv}{^{-1}}
\newcommand{\sinv}{^{\thinspace -1}}

\newcommand{\co}{\colon\thinspace}
\newcommand{\Stt}{\mathfrak{St}}
\newcommand{\Zalg}{\mathbb{Z}\mathrm{-alg}}

\newcommand{\Gr}{\mathrm{Grp}}
\newcommand{\CDem}{\mathfrak{CD}}

\DeclareMathOperator{\Aut}{Aut}
\DeclareMathOperator{\GL}{GL}
\DeclareMathOperator{\Ch}{Ch}
\DeclareMathOperator{\height}{ht}
\DeclareMathOperator{\Id}{Id}

\DeclareMathOperator{\ma}{ma}
\DeclareMathOperator{\pma}{pma}
\DeclareMathOperator{\SL}{SL}
\DeclareMathOperator{\SSL}{E}
\DeclareMathOperator{\Hom}{Hom}
\DeclareMathOperator{\sico}{sc}

\DeclareMathOperator{\C}{C}
\DeclareMathOperator{\GE}{GE_2}
\DeclareMathOperator{\op}{op}
\DeclareMathOperator{\Opp}{Opp}

\numberwithin{equation}{section}

\begin{document}

\renewcommand{\proofname}{{\bf Proof}}

\title{Presentation and uniqueness of Kac--Moody groups over local rings}
\author{Timoth\'ee Marquis}
\address{Université Catholique de Louvain, IRMP, 1348 Louvain-la-Neuve, Belgium}
\email{timothee.marquis@uclouvain.be}
\author{Bernhard M\"uhlherr}
\address{Mathematisches Institut, Justus--Liebig--Universität Gie{\ss}en, 35392 Gie{\ss}en, Germany}
\email{bernhard.m.muehlherr@math.uni-giessen.de}
\thanks{TM is a F.R.S.-FNRS Research associate, and is supported in part by the FWO and the F.R.S.-FNRS under the EOS programme (project ID 40007542).}

\begin{abstract}
To any generalised Cartan matrix (GCM) $A$ and any ring $R$, Tits associated a Kac--Moody group $\mathfrak{G}_A(R)$ defined by a presentation \`a la Steinberg. For a domain $R$ with field of fractions $\mathbb{K}$, we explore the question of whether the canonical map $\varphi_R\colon\thinspace \mathfrak{G}_A(R)\to \mathfrak{G}_A(\mathbb{K})$ is injective. This question for Cartan matrices has a long history, and for GCMs was already present in Tits' foundational papers on Kac--Moody groups.\\
We prove that for any $2$-spherical GCM $A$, the map $\varphi_R$ is injective for all valuation rings $R$ (under an additional minor condition (co)). To the best of our knowledge, this is the first such injectivity result beyond the classical setting.
\end{abstract}

\maketitle

\thispagestyle{empty}
\section{Introduction}

To any reduced root system $\Phi$, one can associate an affine group scheme $\CDem_{\Phi}$ over $\ZZ$, namely the (universal) \emph{Chevalley--Demazure group scheme} of type $\Phi$, such that $\CDem_{\Phi}(\CC)$ is the (universal) complex semisimple algebraic group $G_{\CC}$ with root system $\Phi$; it is characterised by a few simple properties (see \cite{Dem65}), making it the natural analogue of $G_{\CC}$ over any (commutative, unital) ring $R$. The value of $\CDem_{\Phi}$ over a ring $R$ is called a \emph{Chevalley group} over $R$.

When $R$ is a field, Steinberg \cite{St68} proved that $\CDem_{\Phi}(R)$ admits a presentation with generators $\{x_{\alpha}(r) \ | \ \alpha\in\Phi, \ r\in R\}$, subject to a few relations: one first introduces the \emph{Steinberg group} $\mathrm{St}(\Phi,R)$ obtained by ensuring that the sets $\{x_{\alpha}(r) \ | \ r\in R\}$ are copies of $(R,+)$ satisfying certain commutation relations, and one obtains the desired presentation by adding relations corresponding to so-called \emph{Steinberg symbols} (see Remark~\ref{remark:presentation_Chevalley_groups} for precise definitions). Denoting by $G_{\Phi}(R)$ the group defined by this presentation over any ring $R$, Steinberg showed that there is a natural morphism $\varphi_R\co G_{\Phi}(R)\to \CDem_{\Phi}(R)$ which is an isomorphism whenever $R$ is a field.

Since then, determining for which rings $R$ the map $\varphi_R$ is injective (a ring satisfying this property is called \emph{universal} for $\Phi$ in \cite{AMor88}) has been an active topic of research, notably in the context of algebraic $K$-theory: the kernel of $\mathrm{St}(\Phi,R)\to \CDem_{\Phi}(R)$ is denoted $K_2(\Phi,R)$, and the injectivity of $\varphi_R$ amounts to $K_2(\Phi,R)$ being generated by Steinberg symbols. Besides \cite{St68}, a few early milestones regarding this question include the proofs of universality (for all $\Phi$) of the ring of integers (\cite{Mil71}, \cite{HR75}, \cite{Behr75}), of local or even semilocal rings with at most one residue field $\FF_2$ (\cite{Cohn66}, \cite{Ste73}), of polynomial rings $\KK[t]$ and Laurent polynomial rings $\KK[t,t\inv]$ for $\KK$ a field (\cite{Reh75}, \cite{Mor82}), and of rings of the form $\ZZ[1/p_1,\dots,1/p_r]$ for suitable finite sets of primes $p_1,\dots,p_r$ (\cite{AMor88}). As a more recent result, let us for instance mention \cite[Theorem~1.3]{Sin25}, showing that for a Dedekind domain $R$, the polynomial ring $R[t_1,\dots,t_n]$ is universal for some root systems of large rank provided $R$ is. Note that, despite all these positive results, even Euclidean domains such as $\ZZ[1/p]$ for a prime $p\geq 5$ are not universal (see \cite[p.461]{AMor88}).

Around the 1970's, constructions of algebras and groups attached to \emph{generalised Cartan matrices} (namely, the \emph{Kac--Moody algebras} and \emph{Kac--Moody groups}, see \cite{Kac} and \cite{Tits87}) began to emerge, first as (infinite-dimensional) generalisations of semisimple Lie algebras and algebraic groups, and progressively becoming important objects of study with a very rich theory in a variety of domains, including geometric group theory, algebraic geometry, representation theory and theoretical physics (see e.g. \cite{KMGbook} and the references therein).

More precisely, to any generalised Cartan matrix (GCM) $A=(a_{ij})_{i,j\in I}$, Tits associated in \cite{Tits87} a group functor $\G_A$ over the category of rings\footnote{To simplify the exposition, we only mention in this introduction the Kac--Moody group functors of \emph{simply connected type} --- see Remark~\ref{remarkintro:12}(2) concerning the general case.}, defined at each ring $R$ by a presentation \emph{à la Steinberg} (in particular, $\G_A(R)$ coincides with $G_{\Phi}(R)$ for $A$ a Cartan matrix with root system $\Phi$).  On the other hand, Mathieu constructed in \cite{M89} an ind-group scheme $\G^{\pma}_A$, which coincides with $\CDem_{\Phi}$ for $A$ a Cartan matrix with root system $\Phi$ (see \cite[3.8]{Rou16}), and such that for each ring $R$, there is a natural morphism $\varphi_R\co\G_A(R)\to \G^{\pma}_A(R)$. When $R$ is a field, this morphism is injective, but its image is much smaller than $\G^{\pma}_A(R)$: in fact, $\G^{\pma}_A(R)$ is naturally a Hausdorff topological group in which $\G_A(R)$ embeds as a dense subgroup (at least in the generic case). The groups $\G_A(R)$ and $\G^{\pma}_A(R)$ for $R$ a field are then referred to as \emph{minimal} and \emph{maximal} Kac--Moody groups (where ``minimal'' should be understood as ``being generated by the $|I|$ copies of $\SL_2(R)$ attached to the simple roots''). In \cite{Tits87}, Tits also asserted the existence of a group functor associated to $A$ and satisfying a few simple axioms which any reasonable ``minimal Kac--Moody group functor'' ought to satisfy, and proves that the restriction of such a functor to the category of fields is uniquely determined (and coincides with $\G_A$). Such a functor $\G_A^{\min}$ turns out to be unique (at least over domains, see \cite[Proposition~8.129]{KMGbook}): for each ring $R$, the group $\G_A^{\min}(R)$ can be constructed as the subgroup of $\G^{\pma}_A(R)$ generated by the $|I|$ fundamental copies of $\SL_2(R)$ (see \cite[\S8.8]{KMGbook}). In particular, $\varphi_R$ can be viewed as a morphism $$\varphi_R\co\G_A(R)\to \G_A^{\min}(R)\subseteq\G^{\pma}_A(R),$$ or if $R$ is a domain with field of fractions $\KK$, as the canonical map $\G_A(R)\to \G_A(\KK)$ since $\G_A^{\min}(R)\subseteq\G_A^{\min}(\KK)=\G_A(\KK)$.

The question of whether $\G_A$ is the ``good'' minimal Kac--Moody group functor over certain categories of rings beyond fields, or in other words, of knowing for which rings $R$ the map $\varphi_R\co\G_A(R)\to\G_A^{\min}(R)$ is injective, is already very much
 present in Tits' foundational papers (see e.g. \cite[\S6]{Tits85} and \cite[\S3.5]{Tits89}). In this paper, we prove that for any \emph{$2$-spherical} GCM $A=(a_{ij})_{i,j\in I}$ (that is, such that $a_{ij}a_{ji}\leq 3$ for all $i,j\in I$ with $i\neq j$), the map $\varphi_R\co\G_A(R)\to\G_A^{\min}(R)$ is injective for any local ring $R$ that is a Bezout domain (in other words, for any valuation ring), up to a minor technical condition (co) (see below). To the best of our knowledge, this is the first such injectivity result beyond the classical setting of Cartan matrices. Note that the groups $\G_A^{\min}(R)$ over rings $R$ with a discrete valuation have been investigated over the past decade (see for instance \cite{GR14}, \cite{Rou16} and \cite{BHR25}) with the aim of extending the classical Bruhat--Tits theory \cite{BT72} of semisimple algebraic groups over fields with a discrete valuation.
 
To establish our result, we actually prove a more precise statement, namely that $\G_A^{\min}(R)$ can be presented as a Curtis--Tits amalgam: for each subset $J$ of $I$ with $|J|\leq 2$, let $G_{JR}$ denote the Chevalley group $\CDem_{\Phi_J}(R)$, where $\Phi_J$ is the root system with Cartan matrix $A_J=(a_{ij})_{i,j\in J}$. Thus, for each distinct $i,j\in I$, the group $G_{\{i,j\}R}$ is a copy of $\SL_2(R)\times\SL_2(R)$, $\SL_3(R)$, $\mathrm{Sp}_4(R)$ or of $\CDem_{G_2}(R)$, depending on whether $a_{ij}a_{ji}=0,1,2$ or $3$, while $G_{iR}$ and $G_{jR}$ are its rank $1$ subgroups, isomorphic to $\SL_2(R)$. We define the \emph{Curtis--Tits amalgam} $\mathrm{CT}_A(R)$ as the limit of the inductive system\footnote{The embeddings of the rank $1$ subgroups $G_{iR}$ in the rank $2$ subgroups $G_{\{i,j\}R}$ considered in this paper are the ``standard'' ones, namely, those such that the amalgam $\mathrm{CT}_A(R)$ has a presentation with generators $\{x_{\alpha}(r) \ | \ r\in R, \ \alpha\in\bigcup_{|J|=2}\Phi_J\}$ and relations the union over all $J\subseteq I$ with $|J|=2$ of the relations in Steinberg's presentation of $G_{JR}=\CDem_{\Phi_J}(R)$ (assuming $R$ is universal for $\Phi_J$, e.g. $R$ a local ring).} of groups $\{G_{JR} \ | \ J\subseteq I, \ |J|\leq 2\}$. 

Consider the following condition (co) for a ring $R$:
\begin{itemize}
\item[(co)] $R$ has no quotient $\FF_2$ if $a_{ij}a_{ji}=2$ for some $i,j\in I$, and $R$ has no quotient $\FF_2$ or $\FF_3$ if $a_{ij}a_{ji}=3$ for some $i,j\in I$.
\end{itemize} 

Here is our main result.

\begin{thmintro}\label{thmintro:main}
Let $A$ be a $2$-spherical GCM and $R$ be a valuation ring satisfying (co), with field of fractions $\KK$. Then the canonical morphisms $$\mathrm{CT}_A(R)\to\G_A(R)\to \G_A^{\min}(R)\subseteq\G_A(\KK)$$ are isomorphisms. In particular, the map $\G_A(R)\to\G_A(\KK)$ is injective.
\end{thmintro}

The proof of Theorem~\ref{thmintro:main} can be found in Section~\ref{section:TCSofGminR}.

\begin{remarkintro}
Theorem~\ref{thmintro:main} for $R$ a field is the main result of \cite{AbrM97}. More precisely, Abramenko--Mühlherr describe in \cite{AbrM97} two different approaches to show that a group acting chamber-transitively on a thick $2$-spherical twin building $\CCC$ (with an additional condition amounting to the above condition (co)) can be presented as a Curtis--Tits amalgam, as in Theorem~\ref{thmintro:main}. The second approach, detailed in an unpublished preprint \cite{Muh99} by the second author, is of geometric nature:  
it consists in proving that a certain chamber system $\Opp(\CCC)$ associated to $\CCC$ is simply connected (see \S\ref{section:TCSdef} for precise definitions). 

To prove Theorem~\ref{thmintro:main}, we introduce the notion of (simply connected) \emph{twin chamber system} $\CCC$ (see \S\ref{section:TCSdef}), which englobes the twin buildings mentioned above, but also similar objects constructed from $2$-spherical Kac--Moody groups over local rings (see \S\ref{section:TCSofGminR}). We then show, as in \cite{Muh99}, that groups acting chamber-transitively on $\Opp(\CCC)$ admit a presentation as a Curtis--Tits amalgam (see Corollary~\ref{corollary:sctwincs_implies_CTA}). We expect Corollary~\ref{corollary:sctwincs_implies_CTA} to have further applications beyond Theorem~\ref{thmintro:main}.
\end{remarkintro}

\begin{remarkintro}\label{remarkintro:12}
$ $
\begin{enumerate}
\item 
To prove Theorem~\ref{thmintro:main}, we show that the map $\mathrm{CT}_A(R)\to \G_A^{\min}(R)$ is an isomorphism, thereby establishing the desired isomorphism $\G_A(R)\cong \G_A^{\min}(R)$. Of course, this also shows that $\mathrm{CT}_A(R)\cong\G_A(R)$, which recovers a result of Allcock (see \cite[Theorem~1.1(iv) and Corollary~1.3]{All2}) for valuation rings. Note that Allcock's result that $\mathrm{CT}_A(R)\cong\G_A(R)$ when $A$ is $2$-spherical holds for arbitrary rings $R$ satisfying (co); in particular, for such rings, the injectivity of $\G_A(R)\to \G_A^{\min}(R)$ is equivalent to the presentation of $\G_A^{\min}(R)$ as the Curtis--Tits amalgam $\mathrm{CT}_A(R)$.
\item
One can define Kac--Moody groups $\G_{\DDD}(R)$ associated to more general data $\DDD$ than the GCM $A$, called \emph{Kac--Moody root data} --- this corresponds in the classical case to considering the different isogeny types of a semisimple algebraic group. For the purpose of this introduction, we formulated Theorem~\ref{thmintro:main} for the \emph{simply connected} Kac--Moody root datum $\DDD=\DDD_A^{\sico}$ as in \cite[Example~7.11]{KMGbook} (that is, $\G_A=\G_{\DDD_A^{\sico}}$), but we prove Theorem~\ref{thmintro:main} in the setting of arbitrary Kac--Moody root data (see Theorem~\ref{thm:main_amalgamKM}). 
\end{enumerate}
\end{remarkintro}

Note that Theorem~\ref{thmintro:main} provides new universality results even in the classical setting. Indeed, let $\bar{A}$ be a Cartan matrix with irreducible root system $\bar{\Phi}$, and let $A$ be the extended matrix of $\bar{A}$ (see \cite[\S5.3]{KMGbook}): this is a GCM, of so-called untwisted affine type, which is $2$-spherical if $\bar{\Phi}$ is not of type $A_1$. Moreover, there is a Kac--Moody root datum $\DDD$ associated to $A$ for which there is a natural morphism $\varphi_R\co\G_{\DDD}(R)\to \mathfrak{CD}_{\bar{\Phi}}(R[t,t\inv])$ that is an isomorphism whenever $R$ is a field (see \cite[\S7.6]{KMGbook}). Theorem~\ref{thmintro:main} (or rather, Theorem~\ref{thm:main_amalgamKM}) then implies that for any valuation ring $R$ satisfying (co), the Laurent polynomial ring $R[t,t\inv]$ is universal for $\bar{\Phi}$ (see Remark~\ref{remark:proofPropB}), thus generalising Morita's result \cite{Mor82} for such root systems.

\begin{corintro}\label{corintro:Laurentpolynuniversal}
Let $\Phi$ be an irreducible reduced root system, and suppose that $\Phi$ is not of type $A_1$. Let $R$ be a valuation ring satisfying (co). Then $R[t,t\inv]$ is universal for $\Phi$.
\end{corintro}

Finally, note that, as soon as $\SL_2(R)$ is elementary generated (such rings $R$ are called \emph{$\GE$-rings} in \cite{Cohn66}, and include local rings), the natural morphism $\varphi_R\co\G_A(R)\to \G_A^{\min}(R)$ is surjective. On the other hand, even if $R$ is a $\GE$-ring and a domain, with field of fractions $\KK$, the image under $\varphi_R$ of the subgroup $U^+_R$ of $\G_A(R)$ generated by the positive real root groups (see \S\ref{subsection:TCTFGDDD} for precise definitions) is in general properly contained in $U^{\min+}_R:=U^{+}_{\KK}\cap \G_A^{\min}(R)$: this happens for instance as soon as $A$ is not $2$-spherical, as observed by Tits (see \cite[Remark~3.10(d)]{Tits87}). As a byproduct of our methods, we show that $U^+_R$ and $U^{\min+}_R$ nevertheless coincide when $A$ is $2$-spherical and $R$ is a local domain satisfying (co) --- see Theorem~\ref{thm:UbarRisGminRcapUK}(3). 

\begin{propintro}
Let $A$ be a $2$-spherical GCM and $R$ be a local domain satisfying (co), with field of fractions $\KK$. Consider the natural morphism $\varphi_R\co\G_A(R)\to \G_A^{\min}(R)\subseteq\G_A(\KK)$. Then $$\varphi_R(U^{+}_R)=U^{+}_{\KK}\cap \G_A^{\min}(R).$$
\end{propintro}

The paper is structured as follows. After some preliminaries on Kac--Moody groups in Section~\ref{section:preliminaries}, we establish properties of these groups in the three next sections, under the different assumptions on the ring $R$ required for Theorem~\ref{thmintro:main}:  for Bezout domains in Section~\ref{section:BLDOGminOBD}, for rings satisfying the condition (co) in Section~\ref{section:GBSRGIT2SC}, and for local rings in Section~\ref{section:opposition_local_Bezout_domains}. In Section~\ref{section:TCSdef}, after briefly recalling some terminology on chamber systems, we introduce the notion of twin chamber systems, and state our main result about them (Theorem~\ref{thm:simplyconnectedtwincs}). This result, which is of geometric nature and is independent of the Kac--Moody setting, is proved in Section~\ref{section:proofsimpleconnected}. Finally, in Section~\ref{section:TCSofGminR}, we show that $2$-spherical Kac--Moody groups over valuation rings satisfying (co) yield simply connected twin chamber systems.

\subsection*{Acknowledgement} We would like to thank Pierre-Emmanuel Caprace for useful comments on an earlier version of the paper.


\section{Preliminaries}\label{section:preliminaries}

We start by introducing the notations and terminology that will be adopted throughout the paper.

\subsection{Rings}
By a \textbf{ring} we always mean a commutative, unital ring. Given a ring $R$, we denote by $R^{\times}$ the multiplicative group of its units. We denote by $\Zalg$ the category of rings, and by $\Gr$ the category of groups.

\subsection{About \texorpdfstring{$\SL_2$}{SL2}}\label{subsection:aboutSL2}
Let $R$ be a ring. We write $B^{+}_2(R)$ (resp. $B^-_2(R)$) for the subgroup of upper (resp. lower) triangular matrices in $\SL_2(R)$, and $U^{\pm}_2(R)$ for the unipotent matrices of $B^{\pm}_2(R)$. We also let $\SSL_2(R):=\langle U^+_2(R),U^-_2(R)\rangle$ denote the elementary subgroup of $\SL_2(R)$. 

Following \cite{Cohn66}, we call $R$ a {\bf $\mathrm{GE}_2$-ring} if $\SL_2(R)=E_2(R)$. For instance, Euclidean rings and rings of stable rank $1$ (in particular, local rings) are $\mathrm{GE}_2$-rings.

Recall that a {\bf Bezout ring} is a ring in which the sum of two principal ideals is again a principal ideal.

\begin{lemma}\label{lemma:sl2kK_Steinberg}
Assume that $R$ is a Bezout domain, with field of fractions $\KK$. Then $$\SL_2(\KK)=\SL_2(R) B_2^+(\KK).$$
Moreover,
$$U^+_2(\KK)\begin{psmallmatrix}0 &1\\ -1&0\end{psmallmatrix}B^+_2(\KK)=YB^+_2(\KK),$$
where $Y$ is a set of coset representatives for $(\SL_2(R)-B^+_2(R))/B^+_2(R)$.
\end{lemma}
\begin{proof}
Consider a matrix $M=\begin{psmallmatrix}p & q\\ r & s \end{psmallmatrix}$ of $\SL_2(\KK)$. Choose $c,d\in R$ relatively prime such that $cp+dr=0$, and let $a,b\in R$ be such that $ad-bc=1$. Then multiplying $M$ on the left by the matrix $\begin{psmallmatrix}a & b\\ c & d \end{psmallmatrix}$ of $\SL_2(R)$ yields a matrix in $\mathrm{B}_2(\KK)$. This proves the first claim.

Since $Y\cap B_2^+(\KK)=Y\cap (\SL_2(R)\cap B_2^+(\KK))=Y\cap B_2^+(R)=\varnothing$ and $\SL_2(R)=\coprod_{y\in Y\cup\{1\}}yB^+_2(R)$, it follows from the first statement that $\SL_2(\KK)=B_2^+(\KK)\sqcup YB_2^+(\KK)$. The second claim then follows from the Bruhat decomposition $\SL_2(\KK)=B_2^+(\KK)\sqcup U_2^+(\KK)\begin{psmallmatrix}0 &1\\ -1&0\end{psmallmatrix}B^+_2(\KK)$.
\end{proof}

We also recall that over local rings $R$, the group $\SL_2(R)=E_2(R)$ admits the following presentation.
\begin{lemma}\label{lemma:presentation_SL2_localring}
Let $R$ be a local ring. Then $\SL_2(R)$ admits a presentation with generators $\{x_+(r), x_-(r) \ | \ r\in R\}$ and the following relations, for all $a,b\in R$ and $r,s\in R^{\times}$:
\begin{enumerate}
\item
$x_{\pm}(a)x_{\pm}(b)=x_{\pm}(a+b)$,
\item
$\widetilde{s}(r)x_{\pm}(a)\widetilde{s}(r)\inv=x_{\mp}(ar^{\mp 2})$, where $\widetilde{s}(r):=x_+(r)x_-(r\inv)x_+(r)$,
\item
$r^h\cdot s^h=(rs)^h$, where $u^h:=\widetilde{s}(1)\inv \widetilde{s}(u)$ for all $u\in R^{\times}$.
\end{enumerate}
The isomorphism from this presentation to $\SL_2(R)$ is given by $x_+(r)\mapsto \begin{psmallmatrix}1&r\\ 0&1\end{psmallmatrix}$, $x_-(r)\mapsto \begin{psmallmatrix}1&0\\ -r&1\end{psmallmatrix}$.
\end{lemma}
\begin{proof}
This follows from \cite[Theorem~4.1]{Cohn66} and \cite[Corollary to Theorem~1]{Cohn68} (note that Cohn actually gives a slightly different presentation of $\SL_2(R)$, but it is straightforward to check that both presentations are equivalent, see e.g. \cite[Appendix~A]{Hut22}). 
\end{proof}


\subsection{Kac--Moody root systems}
Let $A=(a_{ij})_{i,j\in I}$ be a \textbf{generalised Cartan matrix} (GCM), that is, $A$ is an integral matrix indexed by some finite set $I$, satisfying $a_{ii}=2$, $a_{ij}\leq 0$ and $a_{ij}=0\Leftrightarrow a_{ji}=0$ for all $i,j\in I$ with $i\neq j$. The cardinality of $I$ is called the \textbf{rank} of $A$. 

Let $Q:=\bigoplus_{i\in I}\ZZ\alpha_i$ be the free abelian group on the basis $\Pi:=\{\alpha_i \ | \ i\in I\}$. The \textbf{Weyl group} of $A$ is the subgroup $\WW=\WW(A)$ of $\GL(Q)$ (the $\ZZ$-linear permutations of $Q$) generated by the \textbf{simple reflections} $s_i$ ($i\in I$) defined by
$$s_i\co Q\to Q: \alpha_j\mapsto \alpha_j-a_{ij}\alpha_i.$$
The pair $(\WW,S:=\{s_i \ | \ i\in I))$ is then a Coxeter system, with the order $m_{ij}$ of $s_is_j$ ($i\neq j$) satisfying $m_{ij}=2,3,4,6$ or $\infty$, depending on whether $a_{ij}a_{ji}=0,1,2,3$ or $\geq 4$ (see e.g. \cite[Proposition~4.22]{KMGbook}). We denote by $\ell=\ell_S$ the word metric on $\WW$ with respect to $S$. The matrix $A$ is called \textbf{spherical} if $\WW$ is finite (equivalently, $A$ is a Cartan matrix), and \textbf{$2$-spherical} if $m_{ij}<\infty$ for all $i,j\in I$ with $i\neq j$.

Set $\Phi=\Phi(A):=\WW.\Pi\subseteq Q$. Then $\Phi$ coincides with the set of \emph{real roots} of the Kac--Moody algebra of type $A$, see e.g. \cite[\S~3.5]{KMGbook} (alternatively, $\Phi$ can be $\WW$-equivariantly identified with the set of \emph{roots} or \emph{half-spaces} of the Coxeter complex of $(W,S)$, see \cite[Section~B.4]{KMGbook}). In particular, if $A$ is spherical of rank $2$, then $\Phi$ is a root system of type $A_1\times A_1$, $A_2$, $B_2$ or $G_2$. Setting $$\Phi_{\pm}:=\Phi\cap \pm Q_+\quad\textrm{where}\quad Q_+:=\bigoplus_{i\in I}\NN\alpha_i\subseteq Q,$$ 
we have $\Phi=\Phi_+\cup\Phi_-$. We call an element $\alpha\in \Phi_{+}$ a \textbf{positive root} and $\alpha\in\Phi_-$ a \textbf{negative root}; we then also write $\alpha>0$ or $\alpha<0$ accordingly. The \textbf{height} of $\alpha=\sum_{i\in I}n_i\alpha_i\in\Phi$ is the integer $\height(\alpha):=\sum_{i\in I}n_i$.  If $J\subseteq I$, we set $\Phi_{\pm}(J):=\Phi_{\pm}\cap\bigoplus_{i\in J}\ZZ\alpha_i$.

Two distinct roots $\alpha,\beta\in\Phi$ form a \textbf{prenilpotent pair} if there exist $w,v\in\WW$ such that $\{w\alpha,w\beta\}\subseteq\Phi_+$ and $\{v\alpha,v\beta\}\subseteq\Phi_-$. In that case, the (open) interval $$]\alpha,\beta[_{\NN}:=\{i\alpha+j\beta \ | \ i,j\in\NN_{\geq 1}\}\cap\Phi$$
is finite (see \cite[\S7.4.3]{KMGbook}). If $\Phi$ is spherical, then $\{\alpha,\beta\}$ is prenilpotent if and only if $\beta\neq\pm\alpha$.

\subsection{Kac--Moody root data and tori}\label{subsection:KMRDandtori}
Let $A=(a_{ij})_{i,j\in I}$ be a GCM. A  \textbf{Kac--Moody root datum} associated to $A$ is a quintuple $\DDD=(I,A,\Lambda,(c_i)_{i\in I},(h_i)_{i\in I})$, where $\Lambda$ is a free $\ZZ$-module whose $\ZZ$-dual we denote $\Lambda^{\vee}$, and where the elements $c_i\in\Lambda$ and $h_i\in\Lambda^{\vee}$ satisfy $\langle c_j,h_i\rangle=a_{ij}$ for all $i,j\in I$. For instance, the unique Kac--Moody root datum $\DDD$ such that $\Lambda^{\vee}=\bigoplus_{i\in I}\ZZ h_i$ is called the \textbf{simply connected} Kac--Moody root datum, and is denoted $\DDD_A^{\sico}$ (see \cite[\S7.3.1]{KMGbook}). 

To any Kac--Moody root datum $\DDD$, one can associate a group functor $\T_{\Lambda}\co\Zalg\to\Gr$, called the \textbf{split torus scheme}, defined by $\T_{\Lambda}(R):=\Lambda^{\vee}\otimes_{\ZZ}R^{\times}$ for each ring $R$. Alternatively, $\T_{\Lambda}(R)=\Hom_{\Zalg}(\ZZ[\Lambda],R)\approx\Hom_{\Gr}(\Lambda,R^{\times})$, where the isomorphism $\Lambda^{\vee}\otimes_{\ZZ}R^{\times}\stackrel{\sim}{\to} \Hom_{\Gr}(\Lambda,R^{\times})$ is given by the assignment
$$h\otimes r\mapsto \Big[r^h\co\Lambda\to R^{\times}: \lambda\mapsto r^{h}(\lambda):=r^{\langle \lambda,h\rangle}\Big].$$
For instance, if $\DDD=\DDD_A^{\sico}$, then $\T_{\Lambda}(R)=\langle r^{h_i} \ | \ r\in R^{\times}, \ i\in I\rangle\cong (R^{\times})^{|I|}$ (see \cite[\S7.3.3]{KMGbook}).

\subsection{The constructive Tits functor \texorpdfstring{$\G_{\DDD}$}{GD}}\label{subsection:TCTFGDDD}

Let $\DDD=(I,A,\Lambda,(c_i)_{i\in I},(h_i)_{i\in I})$ be a Kac--Moody root datum. For each $\gamma\in\Phi=\Phi(A)$, we consider a copy $\U_{\gamma}$ of the additive group functor $\Ga\co\Zalg\to\Gr$ (given by $\Ga(R)=(R,+)$), by specifying an isomorphism $$x_{\gamma}\co R\stackrel{\sim}{\to}\U_{\gamma}(R):a\mapsto x_{\gamma}(a)\quad\textrm{for each ring $R$.}$$
For $i\in I$, we also set for short $x_{\pm i}:=x_{\pm\alpha_i}$.

\begin{definition}\label{definition:Steinberg_functor}
The {\bf Steinberg functor} associated to $A$ is the group functor $\Stt_A\co\Zalg\to\Gr$ defined as follows: for any ring $R$, we let $\Stt_A(R)$ denote the quotient of the free product of all $\U_{\gamma}(R)$ for $\gamma\in\Phi$ by the relations
\begin{equation}\label{eqn:commut_rel_remy2}
[x_{\alpha}(a),x_{\beta}(b)]=\prod_{\gamma}x_{\gamma}(C^{\alpha\beta}_{ij}a^ib^j)\quad\textrm{for all prenilpotent pairs $\{\alpha,\beta\}$ and all $a,b\in R$,}  \tag{{\bf R0}}
\end{equation}
where $\gamma=i\alpha+j\beta$ runs through $]\alpha,\beta[_{\NN}$ and the integers $C^{\alpha\beta}_{ij}$ are as in \cite[Proposition~7.43]{KMGbook}.
\end{definition}

\begin{definition}\label{definition:constructiveTitsfunctor}
The {\bf constructive Tits functor of type $\DDD$} (see \cite[Definition~7.47]{KMGbook}) is the group functor $\G_{\DDD}\co \Zalg\to\Gr$ such that, for each ring $R$, the group $\G_{\DDD}(R)$ is the quotient of the free product $\Stt_A(R)*\T_{\Lambda}(R)$ by the following relations, where $i\in I$, $r\in R$, $t\in\T_{\Lambda}(R)$, and where we set $\widetilde{s}_i(r):=x_i(r)x_{-i}(r\inv)x_i(r)$ for $r\in R^{\times}$ and $\widetilde{s}_i:=\widetilde{s}_i(1)$:
\begin{align}\label{eqn:relation_txt}
&t\cdot x_i(r)\cdot t\inv = x_i(t(c_i)r), \tag{{\bf R1}}\\ \label{eqn:relation_sts}
&\widetilde{s}_i\cdot t\cdot \widetilde{s}_i\sinv = s_i(t), \tag{{\bf R2}}\\ \label{eqn:relation_rhi}
&\widetilde{s}_i(r\inv)=\widetilde{s}_i\cdot r^{h_i} \qquad\textrm{for $r\in R^{\times}$,} \tag{{\bf R3}}\\ \label{eqn:relation_sus}
&\widetilde{s}_i\cdot u\cdot \widetilde{s}_i\sinv = s_i^*(u) \quad\textrm{for $u\in \U_{\gamma}(R)$,}\quad \gamma\in\Phi, \tag{{\bf R4}}
\end{align}
where the elements $s_i(t)\in\T_{\Lambda}(R)$ from (\ref{eqn:relation_sts}) and $s_i^*(u)\in\U_{s_i\gamma}(R)$ from (\ref{eqn:relation_sus}) are as in \cite[Definition~7.46]{KMGbook}.
\end{definition}

We set for short $G_R:=\G_{\DDD}(R)$. We can identify the \textbf{root groups} $$U_{\alpha}:=U_{\alpha R}:=\U_{\alpha}(R)\quad\textrm{($\alpha\in\Phi$)}\quad\textrm{and}\quad T_R:=\T_{\Lambda}(R)$$ with their image in $G_R$. We set $$U^{\pm}_R:=\langle U_{\alpha R} \ | \ \alpha\in\Phi_{\pm}\rangle \subseteq G_R, \quad  B^{\pm}_R:=T_RU^{\pm}_R\quad\textrm{and}\quad N_R:=\langle \widetilde{s}_i, \ T_{R} \ | \ i\in I\rangle.$$ 
For $w\in\WW$ with reduced decomposition $w=s_{i_1}\dots s_{i_d}$ ($i_1,\dots,i_d\in I$), we write $$\widetilde{w}:=\widetilde{s}_{i_1}\dots\widetilde{s}_{i_d}\in N_R;$$ as the notation suggests, $\widetilde{w}$ only depends on $w$ (see \cite[Proposition~7.57]{KMGbook}). \\
For each $i\in I$, we also set 
$$U^{\pm}_{(i)R}:=U^{\pm}_R\cap \widetilde{s}_iU^{\pm}_R\widetilde{s}_i\sinv,\quad G_{iR}:=\langle U_{\alpha_iR}, U_{-\alpha_iR}\rangle\quad\textrm{and}\quad P_{iR}^{\pm}:=\langle G_{iR}, B^{\pm}_R\rangle.$$
More generally, if $J\subseteq I$, we set
$$U^{\pm}_{JR}:=\langle U_{\alpha R} \ | \ \alpha\in \Phi_{\pm}(J)\rangle, \quad G_{JR}:=\langle G_{iR} \ | \ i\in J\rangle\quad\textrm{and}\quad P_{JR}^{\pm}:=\langle G_{JR},B_R^{\pm}\rangle.$$
There is a \textbf{Cartan--Chevalley involution} $\omega_R\in\Aut(G_R)$ such that $$\omega_R(U_{\alpha})=U_{-\alpha}\quad\textrm{for all $\alpha\in \Phi$}, \quad \omega_R(T_R)=T_R\quad\textrm{and}\quad \omega_R(\widetilde{s}_i)=\widetilde{s}_i\quad\textrm{for all $i\in I$}.$$

\begin{remark}\label{remark:Tsnormalises}
Let $R$ be a ring and $i\in I$.
\begin{enumerate}
\item The relation (\ref{eqn:relation_rhi}) implies that $\widetilde{s}_i^2=(-1)^{h_i}\in T_R$.
\item
The relation (\ref{eqn:relation_sus}) implies that $\widetilde{w}U_{\alpha}\widetilde{w}\sinv=U_{w\alpha}$ for all $w\in \WW$ and $\alpha\in\Phi$.
\item 
$T_R$ normalises each $\U_{\alpha}(R)$ ($\alpha\in\Phi$) by (2), (\ref{eqn:relation_txt}) and  (\ref{eqn:relation_sts}), and hence also $G_{iR}$ and $U^{\pm}_R$.
\item
Together with (\ref{eqn:relation_sts}), (1) and (3) imply that $U^+_{(i)R}$ is normalised by $T_R$ and $\widetilde{s}_i$.
\end{enumerate}
\end{remark}

\begin{remark}\label{remark:Bruhat_and_Levidec_over_fields}
Let $\KK$ be a field. The group $G_{\KK}$ has the following properties (see e.g. \cite[\S7.4.6, \S B.3 and B.4]{KMGbook}).
\begin{enumerate}
\item
The assignment $\widetilde{s}_i\mapsto s_i$ defines a surjective morphism $N_{\KK}\to \WW$ with kernel $T_{\KK}$. The pairs $(B^+_{\KK},N_{\KK})$ and $(B^-_{\KK},N_{\KK})$ form a \emph{twin BN-pair} for $G_{\KK}$. In particular, $G_{\KK}$ admits \textbf{Bruhat decompositions}
$$G_{\KK}=\coprod_{w\in\WW}B^{\pm}_{\KK}\widetilde{w}B^{\pm}_{\KK}.$$
\item
The group $G_{\KK}$ also admits (refined) \textbf{Birkhoff decompositions}
$$G_{\KK}=\coprod_{w\in\WW}U^{\mp}_{\KK}\widetilde{w}T_{\KK}U^{\pm}_{\KK}.$$
Moreover, if $u_-nu_+=1$ for some $u_{\pm}\in U^{\pm}_{\KK}$ and $n\in N_{\KK}$, then $u_-=u_+=n=1$. In particular, 
\begin{equation}\label{eqn:ABrownProp8.76}
B^{+}_{\KK}\cap U^{-}_{\KK}=B^{-}_{\KK}\cap U^{+}_{\KK}=\{1\}.
\end{equation}
\item
If $i\in I$, one has semidirect decompositions $$P^{\pm}_{i\KK}=T_{\KK}G_{i\KK}\ltimes U^{\pm}_{(i)\KK}\quad\textrm{and}\quad U^{\pm}_{\KK}=U_{\pm\alpha_i\KK}\ltimes U^{\pm}_{(i)\KK}.$$
In particular,
\begin{equation}\label{eqn:umcapPip}
U^{\mp}_{\KK}\cap P^{\pm}_{i\KK}=U_{\mp\alpha_i\KK},
\end{equation}
as $P^{\pm}_{i\KK}=U_{\mp\alpha_i\KK}B^{\pm}_{\KK}\cup\widetilde{s}_iB^{\pm}_{\KK}$.
\end{enumerate}
\end{remark}


\subsection{The Tits functor \texorpdfstring{$\G_{\DDD}^{\min}$}{GDmin}}\label{subsection:GminTitsfunctor}
Let $\DDD=(I,A,\Lambda,(c_i)_{i\in I},(h_i)_{i\in I})$ be a Kac--Moody root datum. In \cite{Tits87}, Tits asserts the existence of a group functor over the category of rings, called a \textbf{Tits functor}, satisfying a small number of natural axioms (the axioms (KMG1)--(KMG5) from \cite{Tits87}), and then shows that, up to an additional nondegeneracy condition, the restriction of such a functor to the category of fields is uniquely determined. Such a Tits functor $\G_{\DDD}^{\min}\co\Zalg\to\Gr$ has been explicitly constructed in \cite[\S8.8]{KMGbook} (see \cite[Proposition~8.128]{KMGbook}). It has the following properties:
\begin{enumerate}
\item[(Gmin1)]
$\G_{\DDD}^{\min}$ comes equipped with group functor morphisms $\varphi_i\co\SL_2\to\G_{\DDD}^{\min}$ ($i\in I$) and $\eta\co\T_{\Lambda}\to \G_{\DDD}^{\min}$ such that for each ring $R$, the group morphism $\eta_R\co\T_{\Lambda}(R)\to \G_{\DDD}^{\min}(R)$ is injective. We then identify $T_R=\T_{\Lambda}(R)$ with a subgroup of $G^{\min}_R:=\G_{\DDD}^{\min}(R)$. 
\item[(Gmin2)]
There is a (unique) group functor morphism $\varphi\co \G_{\DDD}\to \G_{\DDD}^{\min}$ such that for each ring $R$, the restriction of $\varphi_R\co \G_{\DDD}(R)\to \G_{\DDD}^{\min}(R)$ to $T_R$ is the identity, and for each $i\in I$,
$$\varphi_{iR}\begin{psmallmatrix}1&r\\ 0&1\end{psmallmatrix}=\varphi_R(x_i(r))\quad\textrm{and}\quad \varphi_{iR}\begin{psmallmatrix}1&0\\ -r&1\end{psmallmatrix}=\varphi_R(x_{-i}(r)) \quad\textrm{for all $r\in R$.}$$ 
The morphism $\varphi_R$ is injective on each $U_{\gamma R}$ ($\gamma\in\Phi$), and we keep the notations $U_{\gamma R}$, $x_{\gamma}(r)$, $x_{\pm i}(r)$ and $\widetilde{s}_i$ ($i\in I$) for the corresponding objects in $G^{\min}_R$. We also set $\oG_{iR}:=\varphi_R(G_{iR})$. Note that $$\varphi_{iR}\begin{psmallmatrix}0&1\\ -1&0\end{psmallmatrix}=\widetilde{s}_i\quad\textrm{and}\quad \oG_{iR}=\varphi_{iR}(\SSL_2(R))\quad\textrm{for all $i\in I$.}$$
\item[(Gmin3)]
If $R_1\to R_2$ is an injective ring morphism, then the group morphism $\G_{\DDD}^{\min}(R_1)\to\G_{\DDD}^{\min}(R_2)$ is injective.
\item[(Gmin4)]
If $\KK$ is a field, then $\varphi_{\KK}\co \G_{\DDD}(\KK)\to \G_{\DDD}^{\min}(\KK)$ is an isomorphism. We then identify $G_{\KK}$ and $G^{\min}_{\KK}$.
\end{enumerate}

For a ring $R$, we set $$\oU^{\pm}_R:=\varphi_R(U^{\pm}_R)=\langle U_{\alpha R} \ | \ \alpha\in\Phi_{\pm}\rangle \subseteq G^{\min}_R\quad\textrm{and}\quad \oB^{\pm}_R:=\varphi_R(B^{\pm}_R)=T_R\overline{U}^{\pm}_R.$$ 
For $J\subseteq I$, we also set 
$$\oG_{JR}:=\varphi_R(G_{JR})=\langle \oG_{iR} \ | \ i\in J\rangle \quad\textrm{and}\quad \oP_{JR}^{\pm}:=\varphi_R(P_{JR}^{\pm})=\langle \oG_{JR}, \oB^{\pm}_R\rangle.$$

\begin{remark}\label{remark:GE2rings}
$ $
\begin{enumerate}
\item
Note that $\mathrm{GE}_2$-rings are precisely those rings $R$ such that $\varphi_R\co G_R\to G^{\min}_R$ is surjective.
\item
If $R$ is a domain with field of fractions $\KK$, then by (Gmin3) and (Gmin4) we can identify $G^{\min}_R$ with a subgroup of $G_{\KK}$ so that the natural map $G_R\to G_{\KK}$ coincides with $\varphi_R\co G_R\to G^{\min}_R\subseteq G_{\KK}$ (we then also write $\varphi_R\co G_R\to G_{\KK}$ for this map). Thus, if $R$ is in addition a  $\mathrm{GE}_2$-ring, then $G^{\min}_R$ is just the image of $G_R$ in $G_{\KK}$.
\item
If $A$ is spherical and $R$ is a local ring, then $\varphi_R\co G_R\to G^{\min}_R$ is an isomorphism by \cite[Corollary~2.14]{Ste73}. Since the result is only stated for $\DDD=\DDD_A^{\sico}$ in \emph{loc. cit.}, we will also give a simple proof of this fact for local domains in \S\ref{subsection:injectivityTFCG}.
\end{enumerate}
\end{remark}


\section{Bruhat-like decomposition of $G^{\min}_R$ over Bezout domains}\label{section:BLDOGminOBD}


Let $R$ be a domain with field of fractions $\KK$, and assume that $R$ is a $\mathrm{GE}_2$-ring, so that $\varphi_{iR}\co\SL_2(R)\to G^{\min}_R\subseteq G_{\KK}$ has image $\oG_{iR}$ for each $i\in I$. Set for short 
$$\oB_{i}:=\varphi_{iR}(B^{+}_2(R)) \quad\textrm{and}\quad \oG_i:=\oG_{iR}.$$
For each $i\in I$, we fix a set $Y_i=Y_{iR}$ of coset representatives for $(\oG_i-\oB_i)/\oB_i$, so that $$\oG_i=\coprod_{y\in Y_i\cup\{1\}}y\oB_{i}.$$ 

\begin{lemma}\label{lemma:Ybar_representatives}
Let $i\in I$. Then $B^+_{\KK}\cap \oG_{i}=\oB_{i}$. 
\end{lemma}
\begin{proof}
Let $g\in \SL_2(R)\subseteq\SL_2(\KK)$ be such that $\varphi_{iR}(g)=\varphi_{i\KK}(g) \in B^+_{\KK}$, and let us show that $g\in B^+_2(\KK)$ (so that $g\in B^+_2(R)$ and hence $\varphi_{iR}(g)\in\oB_i$, as desired). Otherwise, in view of the Bruhat decomposition $\SL_2(\KK)=B^+_2(\KK)\sqcup B^+_2(\KK)sB^+_2(\KK)$ where $s:=\begin{psmallmatrix}0&1\\ -1&0\end{psmallmatrix}$, we would have $\varphi_{i\KK}(g)\in B^+_{\KK}\widetilde{s}_iB^+_{\KK}$, contradicting the Bruhat decomposition in $G_{\KK}$.
\end{proof}


The following proposition and its proof is a straightforward generalisation of \cite[Theorem~15 p.99 and Corollary~1 p.115]{St68} (see also \cite[5.3]{Tits82}).
\begin{prop}\label{prop:Bezout_Steinberg}
Assume that $R$ is a Bezout domain. Let $w\in\WW$, with reduced decomposition $w=s_{i_1}\dots s_{i_d}$. Then
$$\C^{\min}_R(w):=G^{\min}_R\cap B^+_{\KK}\widetilde{w}B^+_{\KK}=Y_{i_1}\dots Y_{i_d}(G^{\min}_R\cap B^+_{\KK}),$$
with uniqueness of writing on the right-hand side. Moreover,
$$G^{\min}_R=\coprod_{w\in \WW}\C^{\min}_R(w).$$
\end{prop}
\begin{proof}
Set for short $B^{\min+}_R:=G^{\min}_R\cap B^+_{\KK}$. We prove the first claim by induction on $d$.

For $d=0$, there is nothing to prove. Assume now that the claim holds for $d$, and let $w\in\WW$ with reduced decomposition $w=s_{i_0}s_{i_1}\dots s_{i_d}$. Recall from Remark~\ref{remark:Bruhat_and_Levidec_over_fields}(3) that $U^+_{\KK}=U_{\alpha_{i_0}\KK} U^+_{(i_0)\KK}$ and from Remark~\ref{remark:Tsnormalises} that $\widetilde{s}_{i_0}$ normalises $U^+_{(i_0)\KK}$ and $T_{\KK}$. Hence, together with Lemma~\ref{lemma:sl2kK_Steinberg},
$$B^+_{\KK}\widetilde{s}_{i_0}B^+_{\KK}= U_{\alpha_{i_0}\KK}\widetilde{s}_{i_0}B^+_{\KK}= Y_{i_0}B^+_{\KK}.$$
Note also that $B^+_{\KK}\widetilde{w}B^+_{\KK}=B^+_{\KK}\widetilde{s}_{i_0}B^+_{\KK}\widetilde{s}_{i_1}\dots \widetilde{s}_{i_d}B^+_{\KK}$ by \cite[(2) p.320]{BrownAbr}.
Since $G^{\min}_R$ contains $Y_{i_0}$, it then follows from the induction hypothesis that 
\begin{align*}
G^{\min}_R\cap B^+_{\KK}\widetilde{w}B^+_{\KK}&= G^{\min}_R\cap B^+_{\KK}\widetilde{s}_{i_0}B^+_{\KK}\widetilde{s}_{i_1}\dots \widetilde{s}_{i_d}B^+_{\KK} \\
&=G^{\min}_R\cap Y_{i_0}B^+_{\KK} \widetilde{s}_{i_1}\dots \widetilde{s}_{i_d} B^+_{\KK}\\
& = Y_{i_0} (G^{\min}_R\cap B^+_{\KK} \widetilde{s}_{i_1}\dots \widetilde{s}_{i_d} B^+_{\KK})\\
&=Y_{i_0}\dots Y_{i_d} B^{\min+}_R. 
\end{align*}
We also prove the uniqueness of writing by induction on $d$. Assume that $y'_{i_1}\dots y'_{i_d}b'= y_{i_1}\dots y_{i_d}b$ for some $b,b'\in B^{\min+}_{R}$ and some $y_{i_r},y'_{i_r}\in Y_{i_r}$. Then $$y_{i_1}\inv y'_{i_1}\dots y'_{i_d}b'=y_{i_2}\dots y_{i_d}b.$$
Note that $y_{i_1}\inv y'_{i_1}$ either belongs to $B^+_{\KK}$ or to $B^+_{\KK}\widetilde{s}_{i_1}B^+_{\KK}$. But the latter case cannot occur, for otherwise $y_{i_1}\inv y'_{i_1}\dots y'_{i_d}b'\in B^+_{\KK}\widetilde{s}_{i_1}B^+_{\KK}\widetilde{s}_{i_2}\dots \widetilde{s}_{i_d}B^+_{\KK}=B^+_{\KK}\widetilde{w}B^+_{\KK}$, contradicting the fact that $y_{i_2}\dots y_{i_d}b\in B^+_{\KK}\widetilde{s}_{i_1}\widetilde{w}B^+_{\KK}$. Thus $y_{i_1}\inv y'_{i_1}\in B^+_{\KK}\cap \oG_{i_1}=\oB_{i_1}$ by Lemma~\ref{lemma:Ybar_representatives}, so that $y_{i_1}=y'_{i_1}$ because $Y_{i_1}$ is a set of coset representatives for $(\oG_{i_1}-\oB_{i_1})/\oB_{i_1}$.

The second claim follows from the Bruhat decomposition in $G_{\KK}$.
\end{proof}


\section{Generation by simple roots groups in the $2$-spherical case.}\label{section:GBSRGIT2SC}
Throughout this section, we fix a $2$-spherical GCM $A=(a_{ij})_{i,j\in I}$ and a ring $R$.

\subsection{Generators for \texorpdfstring{$U^+_R$}{U+R}}

\begin{lemma}\label{lemma:rank2subsystems}
Let $\alpha\in\Phi_+\setminus\Pi$. Then there exist $i,j\in I$ and $v\in\WW$ with $v\alpha_i,v\alpha_j\in\Phi_+$ such that $\alpha\in \thinspace ]v\alpha_i,v\alpha_j[_{\NN}$.
\end{lemma}
\begin{proof}
Let $\alpha\in\Phi_+\setminus\Pi$ and let $i\in I$ and $w\in\WW$ with $\ell(w)$ minimal such that $\alpha=w\alpha_i$. Since $\alpha\notin\Pi$, $w\neq 1$, and hence there exists $j\in I$ with $j\neq i$ such that $\ell(ws_j)<\ell(w)$ (equivalently, $w\alpha_j<0$, see e.g. \cite[Lemma~4.19]{KMGbook}). Write $w=w_1s_j$ for some $w_1\in\WW$ with $\ell(w_1)=\ell(w)-1$. 

Note that $\{\alpha_i,\alpha_j\}$ cannot be (the set of simple roots of a root system) of type $A_1\times A_1$, for otherwise $\alpha=w_1s_j\alpha_i=w_1\alpha_i$, contradicting the minimality of $w$.

If $ws_j\alpha_i>0$, then we can take $v:=ws_j$ since $v\alpha_j=-w\alpha_j>0$ and $\alpha=w\alpha_i\in \thinspace ]v\alpha_i,v\alpha_j[_{\NN} =w\thinspace ]\alpha_i-a_{ji}\alpha_j,-\alpha_j[_{\NN}$. 

Assume now that $ws_j\alpha_i<0$, so that $w=w_2s_is_j$ with $\ell(w_2)=\ell(w)-2$. Then $\{\alpha_i,\alpha_j\}$ cannot be of type $A_2$, for otherwise $w\alpha_i=w_2s_is_j\alpha_i=w_2\alpha_j$ with $\ell(w_2)<\ell(w)$, contradicting the minimality of $w$. Similarly, $\{\alpha_i,\alpha_j\}$ cannot be of type $B_2$, for otherwise $w\alpha_i=w_2s_is_j\alpha_i=(w_2s_j)\alpha_i$ with $\ell(w_2s_j)<\ell(w)$, again a contradiction. 

Thus $\{\alpha_i,\alpha_j\}$ is of type $G_2$. In this case, we can take $v:=w_2=ws_js_i$. Indeed, $s_is_j\alpha_i\in \thinspace ]\alpha_i,\alpha_j[_{\NN}$ and hence $\alpha=w\alpha_i=vs_is_j\alpha_i\in \thinspace ]v\alpha_i,v\alpha_j[_{\NN}$. Moreover, $v\alpha_j=w_2\alpha_j>0$, for otherwise $w=w_3s_js_is_j$ for some $w_3\in \WW$ with $\ell(w_3)=\ell(w)-3$, and hence $w\alpha_i=w_3s_js_is_j\alpha_i=(w_3s_is_j)\alpha_i$ with $\ell(w_3s_is_j)<\ell(w)$, contradicting the minimality of $w$. Since, in addition, $v\alpha_i=w_2\alpha_i=-ws_j\alpha_i>0$, the claim follows.
\end{proof}

\begin{prop}\label{prop:simplegenerationU+}
Suppose that 
\begin{enumerate}
\item[$(\star)$]
for all $i,j\in I$ with $i\neq j$: $U_{\alpha}\subseteq\langle U_{\alpha_i},U_{\alpha_j}\rangle\leq G_R$ for all $\alpha\in \thinspace ]\alpha_i,\alpha_j[_{\NN}$.
\end{enumerate}
Then $U^+_R=\langle U_{\alpha_i} \ | \ i\in I\rangle \leq G_R$. Moreover, $(\star)$ holds whenever the following condition (co) is satisfied:
\begin{itemize}
\item[(co)] $R$ has no quotient $\FF_2$ if $a_{ij}a_{ji}=2$ for some $i,j\in I$, and $R$ has no quotient $\FF_2$ or $\FF_3$ if $a_{ij}a_{ji}=3$ for some $i,j\in I$.
\end{itemize} 
\end{prop}
\begin{proof}
Assume that $(\star)$ holds. We prove by induction on $\height(\alpha)$ that $U_{\alpha}\subseteq V:=\langle U_{\alpha_i} \ | \ i\in I\rangle$ for all $\alpha\in\Phi_+$. If $\height(\alpha)=1$, this is clear. If $\height(\alpha)>1$, then Lemma~\ref{lemma:rank2subsystems} yields $i,j\in I$ and $v\in\WW$ with $v\alpha_i,v\alpha_j\in\Phi_+$ such that $\alpha\in \thinspace ]v\alpha_i,v\alpha_j[_{\NN}$. By induction hypothesis, $U_{v\alpha_i},U_{v\alpha_j}\subseteq V$. By $(\star)$ and Remark~\ref{remark:Tsnormalises}(2), $\widetilde{v}\sinv U_{\alpha}\widetilde{v}=U_{v\inv\alpha}\subseteq\langle  U_{\alpha_i},U_{\alpha_j}\rangle$, and hence $U_{\alpha}\subseteq \langle U_{v\alpha_i},U_{v\alpha_j}\rangle\subseteq V$, as desired.

The last statement follows from \cite[Lemma~11.1]{All2}.
\end{proof}


\subsection{Rank $1$ Levi decompositions}

\begin{lemma}\label{lemma:BGiisGiB}
Assume that $R$ satisfies (co) in case $A$ is not spherical. Let $i\in I$ and $\epsilon\in\{\pm\}$. Then the following assertions hold:
\begin{enumerate}
\item
$B^{\epsilon}_R\widetilde{s}_iB^{\epsilon}_R= U_{\epsilon\alpha_i}\widetilde{s}_iB^{\epsilon}_R$.
\item
$U^{\epsilon}_R=U_{\epsilon\alpha_i}U^{\epsilon}_{(i)R}$.
\item
$U^{\epsilon}_{(i)R}$ is normalised by $T_R$, $U_{\epsilon\alpha_i}$ and $\widetilde{s}_i$, and intersects $U_{\epsilon\alpha_i}$ trivially.
\end{enumerate}
In particular, $G_{iR}U^{\epsilon}_R=U^{\epsilon}_RG_{iR}$ and $G_{iR}B^{\epsilon}_R=B^{\epsilon}_RG_{iR}$.
\end{lemma}
\begin{proof}
Using the Cartan--Chevalley involution $\omega_R$, it suffices to prove the lemma for $\epsilon=+$.

(1) Note that $B^+_R\widetilde{s}_iB^+_R=U^+_R\widetilde{s}_iB^+_R$ by (\ref{eqn:relation_sts}). By assumption and Proposition~\ref{prop:simplegenerationU+}, $U^+_R$ is generated by root groups $U_{\gamma}$ ($\gamma\in\Phi_+$) such that $\{\gamma,\alpha_i\}$ is a prenilpotent pair. Given such a $\gamma\in\Phi_+\setminus\{\alpha_i\}$, it is then sufficient to show that $U_{\gamma}U_{\alpha_i}\widetilde{s}_iB^+_R\subseteq U_{\alpha_i}\widetilde{s}_iB^+_R$. 
Since $\alpha_i$ is the only root of $\Phi_+$ mapped to a negative root by $s_i$, it follows from (\ref{eqn:commut_rel_remy2}) and (\ref{eqn:relation_sus}) that the commutator $[U_{\alpha_i},U_{\gamma}]$ belongs to $U^+_{(i)R}$, and hence that
$$U_{\gamma}U_{\alpha_i}\widetilde{s}_iB^+_R\subseteq U_{\alpha_i}U_{\gamma}[U_{\gamma},U_{\alpha_i}]\widetilde{s}_iB^+_R\subseteq U_{\alpha_i}\widetilde{s}_iB^+_R,$$ as desired.

(2) The claim is equivalent to $U^+_R\subseteq U_{\alpha_i}\widetilde{s}_iU^+_R\widetilde{s}_i\sinv$, or else to $U^+_R\widetilde{s}_i\subseteq U_{\alpha_i}\widetilde{s}_iU^+_R$. This is in turn equivalent to (1).

(3) The fact that $U^+_{(i)R}$ is normalised by $T_R$ and $\widetilde{s}_i$ follows from Remark~\ref{remark:Tsnormalises}(4). On the other hand, if $\U^{\ma+}_A$ and $\U^{\ma}_{\Delta_+\setminus\{\alpha_i\}}$ are the group functors defined in \cite[Definition~8.41]{KMGbook} (where $\Delta_+\subseteq Q_+$ is the set of positive roots of the Kac--Moody algebra of type $A$, see e.g. \cite[\S3.5]{KMGbook}), then we have a semidirect decomposition $\U^{\ma+}_A(R)=\U_{\alpha_i}(R)\ltimes \U^{\ma}_{\Delta_+\setminus\{\alpha_i\}}(R)$ (\cite[Lemma~8.58(4)]{KMGbook}) and a morphism $\phi\co U^+_R\to \U^{\ma+}_A(R)$ mapping $U_{\alpha_i}$ (bijectively) onto $\U_{\alpha_i}(R)$ and $U^+_{(i)R}$ inside $\U^{\ma}_{\Delta_+\setminus\{\alpha_i\}}(R)$ (see \cite[Definition~8.65 and Exercise~8.66]{KMGbook}). In particular, if $u_i\in U_{\alpha_i}$ and $u\in U^+_{(i)R}$, then by (2) we have $u_iuu_i\inv=u_i'u'$ for some $u_i'\in U_{\alpha_i}$ and $u'\in U^+_{(i)R}$, whereas $\phi(u_i'u')\in \U^{\ma}_{\Delta_+\setminus\{\alpha_i\}}(R)$. Then $\phi(u_i')=1$ and hence $u_i'=1$. This shows that $U^+_{(i)R}$ is normalised by $U_{\alpha_i}$, and the same argument yields that $U^+_{(i)R}$ intersects $U_{\alpha_i}$ trivially.
\end{proof}


\section{Properties of \texorpdfstring{$G_R$}{GR} over local rings}\label{section:opposition_local_Bezout_domains}


\subsection{Setting for Section~\ref{section:opposition_local_Bezout_domains}}

Throughout Section~\ref{section:opposition_local_Bezout_domains}, unless otherwise stated, the GCM $A=(a_{ij})_{i,j\in I}$ is assumed $2$-spherical, and $R$ is a local ring, with maximal ideal $L$ and residue field $k:=R/L$. In particular, $R$ is a $\GE$-ring (see \S\ref{subsection:aboutSL2}). Consider the natural map $$\pi_R\co G_R\to G_k,$$
so that $\pi_R$ is the composition of $\varphi_R\co G_R\to G^{\min}_R$ with $G^{\min}_R\to G^{\min}_k=G_k$. By Lemma~\ref{lemma:presentation_SL2_localring}, we have morphisms
$$\widetilde{\varphi}_{iR}\co\SL_2(R)\to G_R: \begin{psmallmatrix}1&r\\ 0&1\end{psmallmatrix}\mapsto x_i(r), \ \begin{psmallmatrix}1&0\\ -r&1\end{psmallmatrix}\mapsto x_{-i}(r)$$
for each $i\in I$ (that is, such that $\varphi_R\circ\widetilde{\varphi}_{iR}=\varphi_{iR}$), and we set for short $$G_i:=G_{iR}=\widetilde{\varphi}_{iR}(\SL_2(R))\subseteq G_R.$$
Finally, we write $U^{\pm}_{L}$ for the kernel of $\pi_R|_{U^{\pm}_R}\co U^{\pm}_R\to U^{\pm}_k$.


\subsection{The kernel of \texorpdfstring{$\pi_R$}{piR}}

\begin{lemma}\label{lemma:GiULisULGi}
Assume that $R$ satisfies (co) in case $A$ is not spherical. Let $i\in I$. Then $G_iU^{\pm}_L=U^{\pm}_LG_i$.
\end{lemma}
\begin{proof}
Set $U^{\pm}_{(i)L}:=U^{\pm}_{(i)R}\cap\ker\pi_R$.
Recall from Lemma~\ref{lemma:BGiisGiB}(2) that $U^{\pm}_R=U_{\pm\alpha_iR}U^{\pm}_{(i)R}$. Since this decomposition of $U^{\pm}_R$ is mapped under $\pi_R$ onto the corresponding semidirect decomposition $U^{\pm}_k=U_{\pm\alpha_ik}\ltimes U^{\pm}_{(i)k}$ of $U^{\pm}_k$ (see Remark~\ref{remark:Bruhat_and_Levidec_over_fields}(3)), we have
$$U^{\pm}_L=U_{\pm\alpha_iR}U^{\pm}_{(i)R}\cap\ker\pi_R\subseteq U_{\pm\alpha_iR}U^{\pm}_{(i)L}.$$
Since $G_i$ normalises $U^{\pm}_{(i)R}$ and hence also $U^{\pm}_{(i)L}$ by Lemma~\ref{lemma:BGiisGiB}(3), we deduce that $$G_iU^{\pm}_L\subseteq G_iU^{\pm}_{(i)L}=U^{\pm}_{(i)L}G_i\subseteq U^{\pm}_LG_i.$$ Taking inverses yields the reverse inclusion.
\end{proof}

\begin{lemma}\label{lemma:preimageBkinGi}
Let $i\in I$ and $g\in G_i$. If $\pi_R(g)\in\varphi_{ik}(B_2^+(k))$, then $g\in  \widetilde{\varphi}_{iR}(U_2^-(L)B_2^+(R))$.
\end{lemma}
\begin{proof}
Let $h=\begin{psmallmatrix}a&b\\ c&d\end{psmallmatrix}\in \SL_2(R)$ with $\widetilde{\varphi}_{iR}(h)=g$. Consider the canonical map $\pi\co\SL_2(R)\to\SL_2(k)$, so that $\varphi_{ik}\circ\pi=\pi_R\circ \widetilde{\varphi}_{iR}$. By assumption, $\varphi_{ik}\circ\pi(h)=\pi_R(g)\in \varphi_{ik}(B_2^+(k))$. Since $\ker\varphi_{ik}\subseteq \{\pm \Id\}\subseteq B_2^+(k)$, we deduce that $\pi(h)\in B_2^+(k)$, and hence $c\in  L$. In particular, $ad=1+bc\in R^{\times}$ and hence $a\in R^{\times}$. 
The lemma follows as
\begin{equation*}
h=\begin{psmallmatrix}a&b\\ c&d\end{psmallmatrix}=\begin{psmallmatrix}1&0\\ ca\inv&1\end{psmallmatrix}\begin{psmallmatrix}a&b\\ 0&a\inv\end{psmallmatrix}\in U_2^-(L)B_2^+(R). \qedhere
\end{equation*}
\end{proof}

\begin{prop}\label{prop:kerpiR}
Assume that $R$ satisfies (co) in case $A$ is not spherical. Then $\ker\pi_R\subseteq U^{-}_LB^{+}_R$.
\end{prop}
\begin{proof}
Let $g\in \ker\pi_R$. Recall that $G_R$ is generated by $T_R$ and the rank $1$ subgroups $G_i$ ($i\in I$). Let $i_1,\dots,i_d\in I$ be such that $g\in U^{-}_LG_{i_1}\dots G_{i_d}B_R^{+}$ and such that $d\in\NN$ is minimal for this property. Assume for a contradiction that $d\geq 1$. Write $g=u_-g_1\dots g_db_+$ with $g_s\in G_{i_s}$, $u_-\in U^{-}_L$ and $b_+\in B_R^{+}$. Note that $g_s\notin \widetilde{\varphi}_{i_sR}(U_2^-(L)B_2^+(R))$ for any $s=1,\dots,d$ by minimality of $d$, since $G_iU^{-}_L=U^{-}_LG_i$ by Lemma~\ref{lemma:GiULisULGi} and $B^{+}_RG_i=G_iB^{+}_R$ by Lemma~\ref{lemma:BGiisGiB} for any $i\in I$. It follows from Lemma~\ref{lemma:preimageBkinGi} that $\pi_R(g_s)\in \varphi_{ik}(\SL_2(k)-B_2^+(k))$ for all $s=1,\dots,d$. In particular, if we set $Y_{ik}:=\varphi_{ik}(Y_k)$ ($i\in I$) for some set $Y_k$ of coset representatives for $(\SL_2(k)-B^+_2(k))/B^+_2(k)$, then $1=\pi_R(g)\in Y_{i_1k}\dots Y_{i_dk}B^+_k$, contradicting the Bruhat-like decomposition of $G_k$ (see Proposition~\ref{prop:Bezout_Steinberg} applied to $R:=k$).
\end{proof}

\begin{corollary}\label{corollary:oppositionRandk}
Assume that $R$ satisfies (co) in case $A$ is not spherical. Then $$\pi_R\inv(B^-_kB^+_k)= B^{-}_R B^{+}_R.$$
\end{corollary}
\begin{proof}
Proposition~\ref{prop:kerpiR} implies that $\pi_R\inv(B^-_kB^{+}_k)= B^{-}_R \ker\pi_R B^{+}_R\subseteq B^{-}_R B^{+}_R$. 
\end{proof}

\subsection{Injectivity theorem for Chevalley groups}\label{subsection:injectivityTFCG}

Proposition~\ref{prop:kerpiR} allows to recover the fact that $\varphi_R\co G_R\to G_R^{\min}$ is an isomorphism when $A$ is spherical and $R$ is a local domain; since this result, which we will need, is only stated for $\DDD=\DDD_A^{\sico}$ in \cite{Ste73}, we provide here an alternative proof for the benefit of the reader (c.f. Remark~\ref{remark:GE2rings}(3)).

\begin{theorem}\label{theorem:spherical_injectivity}
Let $R$ be a local domain with field of fractions $\KK$. Assume that $A$ is spherical. Then the map $\varphi_R\co G_R\to G_{\KK}$ is injective.
\end{theorem}
\begin{proof}
Let $g\in\ker\varphi_R$. By Remark~\ref{remark:GE2rings}(2), $g$ is then in the kernel of $G_R\to G_R^{\min}$, and hence in the kernel of $\pi_R\co G_R\to G_k$ (which is the composition of $G_R\to G_R^{\min}$ with $G_R^{\min}\to G_k^{\min}=G_k$), where $k$ is the residue field of $R$. Proposition~\ref{prop:kerpiR} then implies that $g\in U_R^-T_RU_R^+$, say $g=u_-tu_+$ with $u_{\pm}\in U_R^{\pm}$ and $t\in T_R$. Hence $1=\varphi_R(g)=\varphi_R(u_-)\varphi_R(t)\varphi_R(u_+)$, with $\varphi_R(u_{\pm})\in U^{\pm}_{\KK}$ and $\varphi_R(t)\in T_{\KK}$. The Birkhoff decomposition in $G_{\KK}$ (see Remark~\ref{remark:Bruhat_and_Levidec_over_fields}(2)) then implies that $u_{\pm},t\in\ker\varphi_R$. As $\varphi_R$ is injective on $U^{\pm}_R$ (see e.g. \cite[Exercise~7.62]{KMGbook}) and on $T_R$, it follows that $g=1$, as desired.
\end{proof}

\begin{remark}\label{remark:presentation_Chevalley_groups}
Let $R$ be an arbitrary ring and assume that $A$ is spherical. Let $\DDD=\DDD_A^{\sico}$ be the simply connected Kac--Moody root datum associated to $A$ (see \S\ref{subsection:KMRDandtori}). Then the presentation of $\G_A(R):=\G_{\DDD}(R)$ from Definition~\ref{definition:constructiveTitsfunctor} can be simplified as follows: $\G_A(R)$ has generators $\{x_{\alpha}(r) \ | \ \alpha\in\Phi, \ r\in R\}$, subject to the following relations, for all $\alpha,\beta\in\Phi$, $i,j\in I$, $a,b\in R$ and $r,s\in R^{\times}$, where we set 
$$\widetilde{s}_i(r):=x_{\alpha_i}(r)x_{-\alpha_i}(r\inv)x_{\alpha_i}(r)\quad  \textrm{and}\quad r^{h_i}:=\widetilde{s}_i(1)\inv\widetilde{s}_i(r\inv):$$
\begin{enumerate}
\item[(U)]
$x_{\alpha}(a)x_{\alpha}(b)=x_{\alpha}(a+b)$,
\item[(C)]
For $\beta\neq\pm\alpha$: 
$$[x_{\alpha}(a),x_{\beta}(b)]=\prod_{\stackrel{\gamma\in ]\alpha,\beta[_{\NN}}{\gamma=i\alpha+j\beta}}x_{\gamma}(C^{\alpha\beta}_{ij}a^ib^j),$$
\item[(T)]
$r^{h_i}s^{h_i}=(rs)^{h_i}$,
\item[(SL2)] 
$\widetilde{s}_i(r) x_{\pm\alpha_i}(a)\widetilde{s}_i(r)\inv=x_{\mp\alpha_i}(ar^{\mp 2})$.
\end{enumerate}
If, moreover, the Dynkin diagram associated to $A$ has no connected component of type $A_1$, then the relations (SL2) can be omitted (see \cite[Theorem~8 on p.66]{St68}). The group defined by the relations (U), (C) and (SL2) is called the \emph{Steinberg group} $\mathrm{St}_A(R)$, while the extra relations (T) are referred to in the literature as the \emph{Steinberg symbols}. If $\Phi$ is the root system of $A$, the kernel of the canonical map $\mathrm{St}_A(R)\to \CDem_{\Phi}(R)$ is denoted $K_2(\Phi,R)$. In the language of algebraic $K$-theory, the injectivity of this map is then equivalent to $K_2(\Phi,R)$ being generated by Steinberg symbols. 
\end{remark}

\subsection{Comparing \texorpdfstring{$\overline{U}^+_R$ and $G^{\min}_R\cap U^+_{\KK}$}{two versions of Umin+R}}

In order to compare $\overline{U}^+_R$ and $G^{\min}_R\cap U^+_{\KK}$ for $R$ a local domain with field of fractions $\KK$ and to prove Theorem~\ref{thm:UbarRisGminRcapUK} below, we will need to briefly introduce the maximal Kac--Moody group functor $\G^{\pma}_{\DDD}$ from \cite[\S8.7]{KMGbook}. This is a group functor over the category of rings defined as an inductive limit of certain affine schemes $\B(w)$ for $w\in\WW$ (with respect to closed immersions $\B(w)\to\B(w')$ whenever $w\leq w'$ in the Bruhat order on $\WW$), and which is constructed from an affine group scheme $\U^{\ma+}_A$ (already mentioned in the proof of Lemma~\ref{lemma:BGiisGiB}(3), and such that $\U^{\ma+}_A(\KK)$ contains $U^+_{\KK}$ for every field $\KK$ --- see \cite[Proposition~8.117]{KMGbook}), the torus group scheme $\T_{\Lambda}$, and copies $\varphi_i(\SL_2)$ of $\SL_2$, one for each $i\in I$ (see \emph{loc. cit.}). 

For each ring $R$, the group $G^{\min}_R$ introduced in \S\ref{subsection:GminTitsfunctor} is in fact defined as the subgroup of $\G^{\pma}_{\DDD}(R)$ generated by $\T_{\Lambda}(R)$ and $\varphi_{iR}(\SL_2(R))$ for all $i\in I$ (see \cite[Definition~8.126]{KMGbook}). 
For $i\in I$, we can define an affine group scheme $\PP^{\ma+}_i=\B(s_i)$ (see \cite[\S 8.7, p.~260]{KMGbook}) such that $\PP^{\ma+}_i(R)=T_R\cdot\varphi_{iR}(\SL_2(R))\cdot\U^{\ma+}_A(R)\subseteq\G_{\DDD}^{\pma}(R)$ for each ring $R$ (in particular, $P^+_{i\KK}\subseteq \PP^{\ma+}_i(\KK)$ for each field $\KK$).

In Lemmas~\ref{lemma:RKGpma} and \ref{lemma:UminRvsUK} below, $A$ is an arbitrary GCM and $R$ an arbitrary ring.

\begin{lemma}\label{lemma:RKGpma}
Assume that $R$ is a domain, with field of fractions $\KK$. Let $i\in I$. Then $$\G^{\pma}_{\DDD}(R)\cap\U_A^{\ma+}(\KK)=\U_A^{\ma+}(R)\quad\textrm{and}\quad \G^{\pma}_{\DDD}(R)\cap\PP_i^{\ma+}(\KK)=\PP_i^{\ma+}(R).$$
\end{lemma}
\begin{proof}
Let $k\in\{R,\KK\}$ and let $\X$ be one of the affine group schemes $\U_A^{\ma+}$ or $\PP_i^{\ma+}$. By construction, $\G^{\pma}_{\DDD}(k)$ is an inductive limit of subsets $\B(w)(k)$ with $w\in\WW$, where each $\B(w)$ is an affine scheme and the natural inclusion $\X(k)\to \B(w)(k)$ (assuming $s_i\leq w$ in the Bruhat order if $\X=\PP_i^{\ma+}$) comes from a closed immersion $\X\to \B(w)$ (see \cite[Definition~8.115]{KMGbook}). 
Let $g\in \G^{\pma}_{\DDD}(R)\cap\X(\KK)$, and let $w\in\WW$ such that $g\in \B(w)(R)$, which we may choose so that $s_i\leq w$. Let $A_1,A_2$ be the $\ZZ$-algebras representing $\B(w)$ and $\X$ respectively, and let $\varphi\co A_1\to A_2$ be the surjective algebra morphism such that the inclusion $\X(k)\to \B(w)(k)$ is given by $$\X(k)\approx \Hom(A_2,k)\to \B(w)(k)\approx \Hom(A_1,k):f\mapsto f\circ\varphi.$$ By assumption, $g\in \X(\KK)\approx \Hom(A_2,\KK)$ is such that $g\circ\varphi\in \Hom(A_1,R)\approx \B(w)(R)\subseteq\B(w)(\KK)$. Hence $g\in \Hom(A_2,R)\approx \X(R)$, as desired. 
\end{proof}

\begin{lemma}\label{lemma:UminRvsUK}
Assume that $R$ is a domain, with field of fractions $\KK$. Let $i\in I$. Then:
\begin{enumerate}
\item
$G^{\min}_R\cap U^+_{\KK}\subseteq G^{\min}_R\cap \U^{\ma+}_A(R)$.
\item
$G^{\min}_R\cap P^+_{i\KK}\subseteq G^{\min}_R\cap \PP^{\ma+}_i(R)$.
\end{enumerate}
Moreover, if $R=\KK$ is a field, then the above inclusions are equalities.
\end{lemma}
\begin{proof}
(1) and (2) follow from Lemma~\ref{lemma:RKGpma}, as $G^{\min}_R\subseteq \G^{\pma}_{\DDD}(R)$, $U^+_{\KK}\subseteq \U^{\ma+}_A(\KK)$ and $P^+_{i\KK}\subseteq \PP^{\ma+}_i(\KK)$. Assume now that $R=\KK$ is a field. Then the equality in (1) follows from \cite[Corollary~8.76]{KMGbook}, and this implies the equality in (2) as $\PP^{\ma+}_i(\KK)=T_{\KK}G_{i\KK}\U^{\ma+}_A(\KK)$ and $P^+_{i\KK}=T_{\KK}G_{i\KK}U^+_{\KK}$.
\end{proof}

\begin{theorem}\label{thm:UbarRisGminRcapUK}
Let $R$ be a local domain with field of fractions $\KK$. Assume that $A$ is $2$-spherical and that $R$ satisfies (co) in case $A$ is not spherical. Then:
\begin{enumerate}
\item
$\oP_{iR}^{\pm}=G_R^{\min}\cap P^{\pm}_{i\KK}$ for all $i\in I$.
\item
$T_R\oU^{\pm}_R=G^{\min}_R\cap B^{\pm}_{\KK}$.
\item
$\oU^{\pm}_R=G^{\min}_R\cap U^{\pm}_{\KK}$.
\end{enumerate}
\end{theorem}
\begin{proof}
(1) Let $g\in G_R$ be such that $\varphi_R(g)\in P^+_{i\KK}$, and let us show that $\varphi_R(g)\in \oP^+_{iR}$ (the case $\varphi_R(g)\in P^-_{i\KK}$ follows by applying the Cartan--Chevalley involution, as $\varphi_R\circ\omega_R=\omega_{\KK}\circ\varphi_R$). We have $\varphi_R(g)\in G^{\min}_R\cap P^+_{i\KK}\subseteq G^{\min}_R\cap\PP^{\ma+}_i(R)$ by Lemma~\ref{lemma:UminRvsUK}(2). In particular, if $k$ is the residue field of $R$, the image of $\varphi_R(g)$ under $\PP^{\ma+}_i(R)\to \PP^{\ma+}_i(k)$  belongs to $G^{\min}_k\cap \PP^{\ma+}_i(k)=P^+_{ik}$ (where the last equality follows from Lemma~\ref{lemma:UminRvsUK}(2) applied to $R:=k$). In other words, $\pi_R(g)\in P^+_{ik}$. Hence $g\in\pi_R\inv(P^+_{ik})\subseteq \ker\pi_R \cdot P^+_{iR}\subseteq U^-_RP^+_{iR}$ by Proposition~\ref{prop:kerpiR}, say $g=u_-p_+$ with $u_-\in U^-_R$ and $p_+\in P^+_{iR}$. Thus $\varphi_R(g)=\varphi_R(u_-)\varphi_R(p_+)\in P^+_{i\KK}$, where $\varphi_R(u_-)\in \oU^-_R\subseteq U^-_{\KK}$ and $\varphi_R(p_+)\in \oP^+_{iR}\subseteq P^+_{i\KK}$ (see Remark~\ref{remark:GE2rings}(2)). As $U^-_{\KK}\cap P^+_{i\KK}=U_{-\alpha_i\KK}$ by (\ref{eqn:umcapPip}) in Remark~\ref{remark:Bruhat_and_Levidec_over_fields}(3), we deduce that $\varphi_R(u_-)\in U_{-\alpha_i\KK}\cap\oU^-_R=U_{-\alpha_iR}$ (this last equality follows from \cite[Theorem~8.51(4)]{KMGbook}) and hence $\varphi_R(g)\in U_{-\alpha_iR}\oP^+_{iR}=\oP^+_{iR}$, as desired.

(2) If $g\in G_R$ is such that $\varphi_R(g)\in B^+_{\KK}$, then $\varphi_R(g)\in \oP_{iR}^{+}=\oG_{iR}T_R\oU^+_{R}$ by (1). Since $\oG_{iR}\cap B^+_{\KK}\subseteq T_R\oU^+_R$ by Lemma~\ref{lemma:Ybar_representatives}, we deduce that $\varphi_R(g)\in T_R\oU^+_R$.

(3) This readily follows from (2) and the fact that $T_{\KK}\cap U^+_{\KK}=\{1\}$ (see Remark~\ref{remark:Bruhat_and_Levidec_over_fields}(2)).
\end{proof}

\subsection{Intersections of parabolic subgroups of opposite sign}
Consider the following condition (Bir), which a ring $R$ may or may not satisfy:
\begin{equation}\label{Bir}
\oB^{+}_R\cap \oU^{-}_R=\oB^{-}_R\cap \oU^{+}_R=\{1\}.\tag{Bir}
\end{equation}
For instance, if $R$ is a domain with field of fractions $\KK$, then $R$ satisfies (\ref{Bir}) as in that case $\oB^{\pm}_R\cap \oU^{\mp}_R\subseteq B^{\pm}_{\KK}\cap U^{\mp}_{\KK}=\{1\}$ (see (\ref{eqn:ABrownProp8.76}) in Remark~\ref{remark:Bruhat_and_Levidec_over_fields}).

We again assume that $A$ is $2$-spherical, and consider a local ring $R$ with residue field $k$. We moreover assume that $R$ satisfies (co) in case $A$ is not spherical.

\begin{lemma}\label{lemma:GJRcappiRinvB}
Let  $J\subseteq I$. Then $G_{JR}\cap \pi_R\inv(B^-_kB^+_k)\subseteq T_RU^-_{JR}U^+_{JR}$.
\end{lemma}
\begin{proof}
Let $g\in G_{JR}\cap \pi_R\inv(B^-_kB^+_k)$. Then $\pi_R(g)\in T_kG_{Jk}\cap B^-_kB^+_k$. Consider the Kac--Moody root datum $\DDD(J):=(J,A_J,\Lambda,(c_i)_{i\in J},(h_i)_{i\in J})$ associated to $A_J:=(a_{ij})_{i,j\in J}$, so that (the image in $G_R$ of)  $\G_{\DDD(J)}(R)$ coincides with $T_RG_{JR}$ and $\G_{\DDD(J)}(k)=T_kG_{Jk}$ (see \cite[p.150]{KMGbook}). 
The Birkhoff decompositions $$T_kG_{Jk}=\G_{\DDD(J)}(k)=\coprod_{w\in \WW_J}U^-_{Jk}\widetilde{w}T_kU^+_{Jk}\quad\textrm{and}\quad G_k=\coprod_{w\in\WW}U^-_k\widetilde{w}T_kU^+_k$$
where $\WW_J:=\langle J\rangle\leq \WW$ (see Remark~\ref{remark:Bruhat_and_Levidec_over_fields}(2)) imply that 
$$\pi_R(g)\in T_kG_{Jk}\cap B^-_kB^+_k=T_kU^-_{Jk}U^+_{Jk}.$$
Corollary~\ref{corollary:oppositionRandk} applied to $\G_{\DDD(J)}(R)$ then implies that $g\in  T_RU^-_{JR}U^+_{JR}$, as desired.
\end{proof}

\begin{lemma}\label{lemma:GicapB-B+}
Let $i\in I$. Then $\oG_{iR}\cap \oB^{-}_R\oB^{+}_R\subseteq U_{-\alpha_iR}U_{\alpha_iR}T_R$.
\end{lemma}
\begin{proof}
If $g\in G_{iR}$ is such that $\varphi_R(g)\in  \oB^{-}_R\oB^{+}_R$, then $\pi_R(g)\in B^-_kB^+_k$ and hence the claim follows from Lemma~\ref{lemma:GJRcappiRinvB} with $J=\{i\}$.
\end{proof}

\begin{prop}\label{prop:intersection_parabolic_opposite_signs}
Let  $J\subseteq I$. Assume that $R$ satisfies (\ref{Bir}). Then $\oP_{JR}^{+}\cap \oP_{JR}^-=T_R\oG_{JR}$.
\end{prop}
\begin{proof}
Note that $\oP_{JR}^{\pm}=\oG_{JR}\oB^{\pm}_R$ by Lemma~\ref{lemma:BGiisGiB}. Thus, it is sufficient to check that $\oB^-_R\cap \oG_{JR}\oB^+_R\subseteq T_R\oG_{JR}$. Let $g\in B^-_R$ such that $\varphi_R(g)\in \oG_{JR}\oB^+_R$, say $\varphi_R(g)=\varphi_R(g_Jb_+)$ for some $g_J\in G_{JR}$ and $b_+\in B^+_R$. Recall that $\pi_R$ is the composition of $\varphi_R$ with $G_R^{\min}\to G_k^{\min}=G_k$. Thus, $\pi_R(g_J)=\pi_R(gb_+\inv)\in B^-_kB^+_k$. Hence $g_J\in T_RU^-_{JR}U^+_{JR}$ by Lemma~\ref{lemma:GJRcappiRinvB}, say $g_J=tu_-u_+$ for some $t\in T_R$ and $u_{\pm}\in U^{\pm}_{JR}$. Therefore, $$\varphi_R(u_-\inv t\inv g)=\varphi_R(u_+b_+)\in \oB^{+}_R\cap \oB^{-}_R=T_R$$ by the condition (\ref{Bir}). Thus $$\varphi_R(g)=\varphi_R(t)\varphi_R(u_-)\varphi_R(u_-\inv t\inv g)\in T_R\varphi_R(u_-)T_R\subseteq T_R\oG_{JR},$$ as desired.
\end{proof}


\section{Twin chamber systems}\label{section:TCSdef}

We now come to the more geometric part of the paper, starting with some preliminaries on chamber systems and simple connectedness, before introducing twin chamber systems and announcing our main result (Theorem~\ref{thm:simplyconnectedtwincs}) about them. We then prove Theorem~\ref{thm:simplyconnectedtwincs} in Section~\ref{section:proofsimpleconnected}, and connect twin chamber systems to Kac--Moody groups over local rings in Section~\ref{section:TCSofGminR}.

\subsection{Chamber systems}

Let $I$ be a set. A \textbf{chamber system} over $I$ is a pair $(\CCC,(\sim_i)_{i\in I})$ where $\CCC$ is a set whose elements are called \textbf{chambers} and where $\sim_i$ is an equivalence relation on $\CCC$ for each $i\in I$. Given $i\in I$ and $c,d\in\CCC$, the chamber $c$ is called \textbf{$i$-adjacent} to $d$ if $c\sim_id$. The chambers $c,d$ are called \textbf{adjacent} if they are $i$-adjacent for some $i\in I$. 

If $(\CCC',(\sim_i)_{i\in I})$ is another chamber system over $I$, then a map $f\co\CCC\to\CCC'$ is called a \textbf{chamber map} if $c\sim_id\implies f(c)\sim_if(d)$ for all $c,d\in\CCC$ and $i\in I$. An \textbf{automorphism} of $(\CCC,(\sim_i)_{i\in I})$ is a bijective chamber map $\CCC\to\CCC$ whose inverse is also a chamber map.

A \textbf{gallery} in $(\CCC,(\sim_i)_{i\in I})$ is a finite sequence $(c_0,c_1,\dots,c_k)$ such that $c_{\mu}\in\CCC$ for all $\mu\in\{0,\dots,k\}$ and such that $c_{\mu-1}$ is adjacent to $c_{\mu}$ for all $\mu=1,\dots,k$. The number $k$ is called the \textbf{length} of the gallery. Given a gallery $G=(c_0,c_1,\dots,c_k)$, we put $\alpha(G)=c_0$ and $\omega(G)=c_k$. If $G$ is a gallery and if $c,d\in\CCC$ are such that $c=\alpha(G)$ and $d=\omega(G)$, we say that $G$ is a \textbf{gallery from $c$ to $d$} or that \textbf{$G$ joins $c$ and $d$}. The chamber system is said to be \textbf{connected} if for any two chambers there exists a gallery joining them. A gallery $G$ is \textbf{closed} if $\alpha(G)=\omega(G)$. A gallery $G=(c_0,c_1,\dots,c_k)$ is \textbf{simple} if $c_{\mu-1}\neq c_{\mu}$ for all $\mu\in\{1,\dots,k\}$.

Let $G=(c_0,c_1,\dots,c_k)$ be a gallery. The \textbf{reduced length} of $G$ is the number
$$k^*=|\{1\leq \mu\leq k \ | \ c_{\mu-1}\neq c_{\mu}\}|.$$
We define the $k^*$-tuple $\lambda(G)=(\lambda_1,\dots,\lambda_{k^*})$ and the \textbf{reduced gallery} $G^*=(c_0^*,\dots,c_{k^*}^*)$ of $G$ as follows: we put $c_0^*=c_0$ and $\lambda_0=0$, and define recursively
$$\lambda_{\mu}=\min\{\nu \ | \ \lambda_{\mu-1}<\nu, \ c_{\nu-1}\neq c_{\nu}\}\quad\textrm{and}\quad c_{\mu}^*=c_{\lambda_{\mu}}\quad\textrm{for $\mu=1,\dots,k^*$}.$$
We call $\lambda(G)$ the \textbf{$\lambda$-tuple of $G$}.

Given a gallery $G=(c_0,c_1,\dots,c_k)$, we denote by $G\inv$ the gallery $(c_k,c_{k-1},\dots,c_0)$, and if $H=(c_0',c_1',\dots,c_l')$ is a gallery such that $\omega(G)=\alpha(H)$, then $GH$ denotes the gallery $(c_0,c_1,\dots,c_k=c_0',c_1',\dots,c_l')$.

Let $J$ be a subset of $I$. A \textbf{$J$-gallery} is a gallery $G=(c_0,c_1,\dots,c_k)$ such that for each $\mu\in\{1,\dots,k\}$ there exists an index $j\in J$ with $c_{\mu-1}\sim_jc_{\mu}$. Given two chambers $c,d$, we say that $c$ is \textbf{$J$-equivalent} to $d$ if there exists a $J$-gallery joining $c$ and $d$ and we write $c\sim_Jd$ in this case. Note that $c,d$ are $i$-adjacent if and only if they are $\{i\}$-equivalent. Given a chamber $c$ and a subset $J$ of $I$, the set $R_J(c):=\{d\in\CCC \ | \ c\sim_Jd\}$ is called the \textbf{$J$-residue} of $c$.

\subsection{Homotopy of galleries and simple connectedness}

In the theory of chamber systems there is the notion of \emph{$m$-homotopy} and \emph{$m$-simple connectedness} for each $m\in\NN$. In this paper we are only concerned with the case $m=2$. Therefore our definitions are always to be understood as a specialisation of the general theory to the case $m=2$.

Let $(\CCC,(\sim_i)_{i\in I})$ be a chamber system over a set $I$. Two galleries $G=(c_0,\dots,c_k)$ and $H=(c_0',\dots,c'_{k'})$ are said to be \textbf{elementary homotopic} if there exist $\mu,\nu,\mu',\nu'$ with $0\leq\mu\leq\nu\leq k$ and $0\leq\mu'\leq\nu'\leq k'$ such that the following holds:
\begin{enumerate}
\item[(H1)] $\mu=\mu'$ and $c_{\eta}=c'_{\eta}$ for all $\eta\in\{0,\dots,\mu\}$.
\item[(H2)] $k-\nu=k'-\nu'$ and $c_{k-\eta}=c'_{k'-\eta}$ for all $\eta\in\{0,\dots,k-\nu\}$.
\item[(H3)] The galleries $(c_{\mu},\dots,c_{\nu})$ and $(c'_{\mu'},\dots,c'_{\nu'})$ are $J$-galleries for some subset $J$ of $I$ with $|J|\leq 2$.
\end{enumerate}

Two galleries $G,H$ are said to be \textbf{homotopic} if there exists a sequence $G=G_0,G_1,\dots,G_l=H$ of galleries such that $G_{\mu-1}$ is elementary homotopic to $G_{\mu}$ for all $\mu=1,\dots,l$.

If two galleries $G,H$ are homotopic, then by definition $\alpha(G)=\alpha(H)$ and $\omega(G)=\omega(H)$. A closed gallery $G$ is said to be \textbf{null-homotopic} if it is homotopic to the gallery $(\alpha(G))$. The chamber system $(\CCC,(\sim_i)_{i\in I})$ is called \textbf{simply connected} if it is connected and if each closed gallery is null-homotopic.

\subsection{Twin chamber systems}

\begin{definition}
An {\bf opposition datum} $\CCC$ over a set $I$ is the collection of a pair $(\CCC_+,(\sim_i)_{i\in I})$, $(\CCC_-,(\sim_i)_{i\in I})$ of chamber systems over $I$, together with a symmetric relation $\op\subseteq (\CCC_+\times\CCC_-)\cup(\CCC_-\times\CCC_+)$, called {\bf opposition}. We call $\CCC$ a {\bf twin chamber system} if it satisfies the following axioms, for each $\epsilon\in\{\pm\}$:
\begin{enumerate}
\item[(TCS1)]
Let $c,d\in\CCC_{\epsilon}$ with $c\sim_i d$ for some $i\in I$, and let $x,y\in\CCC_{-\epsilon}$ with $x\sim_j y$ for some $j\in I$. If $d \op x \op c \op y$, then either $d\op y$ or $j=i$.
\item[(TCS2)]
Let $c,d\in\CCC_{\epsilon}$ with $c\sim_i d$ for some $i\in I$, and let $x\in\CCC_{-\epsilon}$ such that $x \op c$. Then there exists $y\in\CCC_{-\epsilon}$ with $x\sim_i y$ such that $c \op y \op d$.
\item[(TCS3)]
For all $c\in\CCC_{\epsilon}$ and all $J$-residues $R_J\subseteq \CCC_{-\epsilon}$ with $|J|\leq 2$, the sets $c^{\op}:=\{x\in\CCC_{-\epsilon} \ | \ c \op x\}$ and $c^{\op}\cap R_J$ are connected.
\item[(TCS4)]
For all $c\in\CCC_{\epsilon}$, there exists a chamber map $\omega_c\co \CCC_{-\epsilon}\to\CCC_{\epsilon}$ such that $\omega_c(x) \op x$ for all $x\in\CCC_{-\epsilon}$, and such that $\omega_c(c^{\op})=\{c\}$.
\end{enumerate}
We call $\CCC$ {\bf simply connected} if both $(\CCC_+,(\sim_i)_{i\in I})$ and $(\CCC_-,(\sim_i)_{i\in I})$ are simply connected. 

We define the chamber system $\Opp(\CCC):=\{(x,y)\in\CCC_+\times\CCC_- \ | \ x \op y\}$ over $I$ by declaring, for each $i\in I$, that $(x,y)\sim_i (x',y')$ if and only if $x\sim_i x'$ and $y\sim_i y'$. 
\end{definition}

\begin{remark}\label{remark:twin_buildings_are_TCS}
Although we will not need this fact, we note that thick $2$-spherical twin buildings are prototypical examples of simply connected twin chamber systems. More precisely, let $\CCC=(\CCC_+,\CCC_-)$ be a thick twin building, with associated codistance $\delta^*$, in the sense of \cite[Definition~5.133]{BrownAbr}. Thus $\CCC_{\pm}$ is the chamber system of a building $(\Delta^{\pm},\delta^{\pm})$ in the sense of \cite[Definition~5.1.1]{BrownAbr}, and $\CCC$ is an opposition datum, with opposition relation $c\op d \Leftrightarrow \delta^*(c,d)=1$. 

It easily follows from the twin buildings axioms that $\CCC$ satisfies (TCS1) and (TCS2). One also checks that $\CCC$ satisfies (TCS4): for $\epsilon\in\{\pm\}$ and $c\in\CCC_{\epsilon}$, choosing a twin apartment $(\Sigma^{\epsilon},\Sigma^{-\epsilon})$ with $c\in\Sigma^{\epsilon}$, there is for each chamber $x\in \CCC_{-\epsilon}$ a unique chamber $\omega_c(x)\in \Sigma^{\epsilon}$ such that $\delta^*(c,x)=\delta^{\epsilon}(c,\omega_c(x))$, and this yields a map $\omega_c\co\CCC_{-\epsilon}\to\Sigma^{\epsilon}\subseteq\CCC^{\epsilon}$ with the desired properties (see \cite[Corollary~5.141(1)]{BrownAbr}). Assume now that $\CCC$ is $2$-spherical (i.e. every rank $2$ residue of $\CCC_{\pm}$ is a spherical building) and that $\CCC_{\pm}$ does not contain any rank $2$ residue isomorphic to one of the buildings associated to the (twisted) Chevalley groups $B_2(\FF_2)$, $G_2(\FF_2)$, $G_2(\FF_3)$ or $^2F_4(\FF_2)$. Then $\CCC$ also satisfies (TCS3) (see e.g. \cite[Remark~5.212]{BrownAbr}). Finally, $\CCC$ is simply connected by \cite[Theorem~4.3]{Ron89}.
\end{remark}

We will prove in Section~\ref{section:proofsimpleconnected} the following theorem.
\begin{theorem}\label{thm:simplyconnectedtwincs}
Let $\CCC$ be a twin chamber system. If $\CCC$ is simply connected, then so is $\Opp(\CCC)$.
\end{theorem}

\begin{definition}
Let $\CCC$ be an opposition datum over $I$. An {\bf automorphism} of $\CCC$ is an automorphism of both $(\CCC_+,(\sim_i)_{i\in I})$ and $(\CCC_-,(\sim_i)_{i\in I})$ preserving the opposition relation $\op$. We write $\Aut(\CCC)$ for the group of automorphisms of $\CCC$. Note that $\Aut(\CCC)$ also acts on $\Opp(\CCC)$. We say that a group $G$ \textbf{acts transitively on $\CCC$} if there is a group morphism $G\to\Aut(\CCC)$ whose image acts transitively on $\Opp(\CCC)$.
\end{definition}

By the general theory of groups acting transitively on chamber systems (see for instance \cite[Proposition~6.5.2]{Sch95}), Theorem~\ref{thm:simplyconnectedtwincs} has the following corollary.

\begin{corollary}\label{corollary:sctwincs_implies_CTA}
Let $\CCC$ be a simply connected twin chamber system, let $(c_+,c_-)\in\Opp(\CCC)$, and let $G$ be a group acting transitively on $\CCC$. Let $\mathcal J_{2}$ be the set of subsets of $I$ of size at most $2$, and for each $J\in\mathcal J_2$, let $G_J$ denote the set of $g\in G$ stabilising the $J$-residues of $c_+$ and $c_-$. Then $G$ is the amalgamated product of the subgroups $G_J$ where $J$ runs over $\mathcal J_2$.
\end{corollary}

Note that Corollary~\ref{corollary:sctwincs_implies_CTA} when $\CCC$ is a thick $2$-spherical twin building with no rank $2$ residue isomorphic to one of the buildings associated to the groups $B_2(\FF_2)$, $G_2(\FF_2)$, $G_2(\FF_3)$ or $^2F_4(\FF_2)$ (cf. Remark~\ref{remark:twin_buildings_are_TCS}) is the main result of \cite{AbrM97}.


\section{Proof of Theorem~\ref{thm:simplyconnectedtwincs}}\label{section:proofsimpleconnected}

Let $\CCC$ be a simply connected twin chamber system over $I$, consisting of the pair $(\CCC_+,(\sim_i)_{i\in I})$, $(\CCC_-,(\sim_i)_{i\in I})$ of chamber systems and of the opposition relation $\op$. 

For $c=(x,y)\in\Opp(\CCC)$, we write $c^+:=x$ and $c^-:=y$. Given a gallery $G=(c_0,c_1,\dots,c_k)$ in $\Opp(\CCC)$ and $\epsilon\in\{\pm\}$, we denote by $G_{\epsilon}$ the gallery $(c_0^{\epsilon},c_1^{\epsilon},\dots,c_k^{\epsilon})$ of $(\CCC_{\epsilon},(\sim_i)_{i\in I})$.

\begin{lemma}\label{lemma:TCS2translated}
Let $J\subseteq I$. Let $c=(c^+,c^-)\in\Opp(\CCC)$, and let $d_1,d_2\in R_J(c^+)$ with $d_1\sim_id_2$ for some $i\in J$. Then there exists $y\in R_J(c^-)$ such that $d_1 \op y \op d_2$.
\end{lemma}
\begin{proof}
Let $(c^+=c_0,c_1,\dots,c_{k-1}=d_1,c_k=d_2)$ be a $J$-gallery, and let $i_{\mu}\in J$ such that $c_{\mu-1}\sim_{i_{\mu}}c_{\mu}$  for each $\mu=1,\dots,k$ (so that $i_k=i$). Set $y_0:=c^-$, so that $y_0 \op c_0$. Using (TCS2) repeatedly, we can inductively construct a sequence of chambers $y_0,y_1,\dots,y_k$ in $\CCC_-$ such that $y_{\mu-1}\sim_{i_{\mu}}y_{\mu}$ and $c_{\mu-1} \op y_{\mu} \op c_{\mu}$ for all $\mu=1,\dots,k$. We can then set $y:=y_k$.
\end{proof}
 
\begin{lemma}\label{lemma:embeddingC-inOpp}
Let $c_{\pm}\in\CCC_{\pm}$. Then the maps $$\CCC_-\to\Opp(\CCC):x\mapsto (\omega_{c_+}(x),x)\quad\textrm{and}\quad \CCC_+\to\Opp(\CCC):x\mapsto (x,\omega_{c_-}(x))$$ are isomorphisms onto their image.
\end{lemma}
\begin{proof}
This is immediate from (TCS4).
\end{proof}

\begin{lemma}
The chamber system $\Opp(\CCC)$ is connected.
\end{lemma}
\begin{proof}
Let $(c_+,c_-),(d_+,d_-)\in\Opp(\CCC)$. By Lemma~\ref{lemma:embeddingC-inOpp}, the map $\CCC_-\to\Opp(\CCC):x\mapsto (\omega_{c_+}(x),x)$ is an isomorphism onto its image. As $\CCC_-$ is connected, there is a gallery in $\Opp(\CCC)$ joining $(c_+,c_-)=(\omega_{c_+}(c_-),c_-)$ and $(d_+':=\omega_{c_+}(d_-),d_-)$. Again by Lemma~\ref{lemma:embeddingC-inOpp}, the map $\CCC_+\to\Opp(\CCC):x\mapsto (x,\omega_{d_-}(x))$ is an isomorphism onto its image. Since $\CCC_+$ is connected, there is a gallery in $\Opp(\CCC)$ joining  $(d_+',d_-)=(d_+',\omega_{d_-}(d_+'))$ and $(d_+,\omega_{d_-}(d_+))=(d_+,d_-)$, as desired.
\end{proof}

\begin{prop}\label{prop:null-homotopicccc}
Let $c\in\CCC_+$, and let $G=(c_0,c_1,\dots,c_k)$ be a closed gallery in $\Opp(\CCC)$ such that $c^+_{\mu}=c$ for all $\mu=0,\dots,k$. Then $G$ is null-homotopic.
\end{prop}
\begin{proof}
By Lemma~\ref{lemma:embeddingC-inOpp}, the map $\pi\co\CCC_-\to\Opp(\CCC):x\mapsto (\omega_{c}(x),x)$ is an isomorphism onto its image. By assumption, $G=\pi(G_-)$. As $\CCC_-$ is simply connected, $G_-$ is null-homotopic in $\CCC_-$, and hence $G=\pi(G_-)$ is null-homotopic in $\Opp(\CCC)$.
\end{proof}

\begin{lemma}\label{lemma:stuttering}
Let $G$ be a gallery in $\Opp(\CCC)$. Then $G$ is homotopic to a gallery $H=(c_0,\dots,c_k)$ such that for each $\mu\in\{1,\dots,k\}$ there exists an $\epsilon\in\{\pm\}$ with the property that $c_{\mu-1}^{\epsilon}=c_{\mu}^{\epsilon}$.
\end{lemma}
\begin{proof}
Reasoning inductively on the length of $G$, we may assume that $G=(d_0,d_1)$. Let $i\in I$ such that $d_0^{\pm}\sim_i d_1^{\pm}$. By (TCS2), there exists $y\in\CCC_-$ with $y\sim_id_0^-$ (and hence also $y\sim_id_1^-$) such that $d_0^+ \op y \op d_1^+$. We can then take $H=((d_0^+,d_0^-),(d_0^+,y),(d_1^+,y),(d_1^+,d_1^-))$.
\end{proof}

\begin{lemma}\label{lemma:x0x1x2}
Let $i\in I$, let $c,d\in\CCC_+$ with $c\sim_i d$, and let $(x_0,x_1,x_2)$ be a gallery in $c^{\op}$ such that $x_0 \op d \op x_2$. Then there exists a gallery $(x_2=y_0,y_1,\dots,y_k=x_0)$ in $d^{\op}$ such that the gallery 
$$((c,x_0),(c,x_1),(c,x_2),(d,y_0),\dots,(d,y_k),(c,x_0))$$
is null-homotopic in $\Opp(\CCC)$. 
\end{lemma}
\begin{proof}
By (TCS1), either $d\op x_1$ or $x_2\sim_ix_1\sim_ix_0$. In the former case, we can choose $k=2$ and $(y_0,y_1,y_2)=(x_2,x_1,x_0)$. In the latter case, we can choose $k=1$ and $(y_0,y_1)=(x_2,x_0)$.
\end{proof}

\begin{lemma}\label{lemma:cdnullhomotopic}
Let $i\in I$, let $c,d\in\CCC_+$ with $c\sim_i d$, and let $(x_0,\dots,x_k)$ be a gallery in $c^{\op}$ such that $x_0 \op d \op x_k$. Then there exists a gallery $(x_k=y_0,y_1,\dots,y_l=x_0)$ in $d^{\op}$ such that the gallery 
$$((c,x_0),\dots,(c,x_k),(d,y_0),\dots,(d,y_l),(c,x_0))$$
is null-homotopic in $\Opp(\CCC)$. 
\end{lemma}
\begin{proof}
We proceed by induction on $k$. If $k=1$, the assertion is trivial; if $k=2$, the assertion follows from Lemma~\ref{lemma:x0x1x2}. Assume now that $k>2$. 

If $x_1 \in d^{\op}$ then applying the induction hypothesis to the gallery $G_1=(x_1,\dots,x_k)$ yields a gallery $H_1=(y_0=x_k,\dots,y_{l_1}=x_1)$ with the required properties. The claim then follows by putting $l=l_1+1$ and $y_l=x_0$. Similarly, if $x_2\in d^{\op}$, the induction hypothesis and the case $k=2$ yield the claim.

Assume now that $x_1,x_2\notin d^{\op}$. By (TCS1), we have $x_0\sim_ix_1$. By (TCS2), there exists $z\in c^{\op}\cap d^{\op}$ with $z\sim_ix_2$ (in particular, $z\neq x_2$). Applying the induction hypothesis to the gallery $(z,x_2,\dots,x_k)$, we obtain a gallery $(y_0=x_k,y_1,\dots,y_m=z)$ in $d^{\op}$ such that $((c,z),(c,x_2),\dots,(c,x_k),(d,x_k=y_0))$ and $((c,z),(d,z=y_m),(d,y_{m-1}),\dots,(d,y_0=x_k))$ are homotopic.

Let $j\in I$ be such that $x_1\sim_j x_2$ and set $J=\{i,j\}$. Since $x_0\sim_ix_1\sim_jx_2\sim_iz$, (TCS3) yields a $J$-gallery $(z=y_0',\dots,y'_{l'}=x_0)$ contained in $d^{\op}$. Since $((c,x_0),(c,x_1),(c,x_2),\dots,(c,x_k),(d,x_k=y_0))$ is elementary homotopic to $$((c,x_0=y'_{l'}),(d,y'_{l'}),\dots,(d,y'_0=z),(c,z),(c,x_2),\dots,(c,x_k),(d,x_k=y_0)),$$ which is homotopic to $$((c,x_0=y'_{l'}),(d,y'_{l'}),\dots,(d,y'_0=z),(c,z),(d,z=y_m),(d,y_{m-1}),\dots,(d,y_0=x_k))$$ and hence to
$$((c,x_0=y'_{l'}),(d,y'_{l'}),\dots,(d,y'_0=z=y_m),(d,y_{m-1}),\dots,(d,y_0=x_k)),$$
the gallery $(y_0,\dots,y_m=y_0',\dots,y'_{l'}=x_0)$ has the required properties.
\end{proof}

\begin{prop}\label{prop:mainwithJ}
Let $J\subseteq I$ be of cardinality at most $2$ and let $c=(c^+,c^-)\in\Opp(\CCC)$. Let $G=(c_0,\dots,c_k)$ be a closed gallery in $\Opp(\CCC)$ with $c_0=c=c_k$ and suppose that $c_{\mu}^+\in R_J(c^+)$ for all $\mu=0,\dots,k$. Then $G$ is null-homotopic.
\end{prop}
\begin{proof}
By Lemma~\ref{lemma:stuttering}, we can assume that for each $\mu\in\{1,\dots,k\}$, there exists an $\epsilon\in\{\pm\}$ such that $c_{\mu-1}^{\epsilon}=c_{\mu}^{\epsilon}$. Let $k^*$ be the reduced length of $G_+$, let $G_+^*=(c_0^*,\dots,c_{k^*}^*)$ and let $\lambda=(\lambda_1,\dots,\lambda_{k^*})$ be the $\lambda$-tuple of $G_+$. In particular, setting $\lambda_0:=0$, $$c_{\mu-1}^*=c^+_{\lambda_{\mu-1}}=c^+_{\lambda_{\mu}-1}\quad\textrm{for $\mu=1,\dots,k^*$,}\quad\textrm{and}\quad c^*_{k^*}=c^+_{\lambda_{k^*}}=c^+_k=c^+_0=c^+.$$

For each $\mu\in\{1,\dots,k^*\}$, we can choose by Lemma~\ref{lemma:TCS2translated}  some $d_{\mu}\in R_J(c^-)$ such that $$c_{\mu}^* \op d_{\mu} \op c_{\mu-1}^*.$$ Put $d_0=d_{k^*+1}=c^-$. 

For $\mu\in\{0,\dots,k^*\}$, since $d_{\mu},d_{\mu+1}\in (c_{\mu}^*)^{\op}$, we can choose by (TCS3) a $J$-gallery $Z_{\mu}=(d_{\mu}=z_{\mu 0},\dots,z_{\mu m_{\mu}}=d_{\mu+1})$ contained in $(c_{\mu}^*)^{\op}$. Let $\overline{Z}_{\mu}$ denote the gallery $$((c^*_{\mu},d_{\mu}=z_{\mu 0}),(c^*_{\mu},z_{\mu 1}),\dots,(c^*_{\mu},z_{\mu m_{\mu}}=d_{\mu+1}))$$ in $\Opp(\CCC)$.

For $\mu\in\{1,\dots,k^*\}$, since $c^+_{\lambda_{\mu}}=c^*_{\mu}\neq c^*_{\mu-1}=c^+_{\lambda_{\mu}-1}$, the assumption made at the beginning of the proof implies that $c^-_{\lambda_{\mu}}=c^-_{\lambda_{\mu}-1}=:d^*_{\mu}$.

For $\mu\in\{1,\dots,k^*\}$, since $d^*_{\mu}=c^-_{\lambda_{\mu}}$ and $d_{\mu}\in (c_{\mu}^*)^{\op}$, we can choose by (TCS3) a gallery $X_{\mu}=(d_{\mu}=x_{\mu 0},\dots,x_{\mu k_{\mu}}=d^*_{\mu})$ contained in $(c_{\mu}^*)^{\op}$. Let $\overline{X}_{\mu}$ denote the gallery $$((c^*_{\mu},d_{\mu}=x_{\mu 0}),(c^*_{\mu},x_{\mu 1}),\dots,(c^*_{\mu},x_{\mu k_{\mu}}=d^*_{\mu}))$$ in $\Opp(\CCC)$. 

For $\mu\in\{1,\dots,k^*\}$, since $d^*_{\mu}=c^-_{\lambda_{\mu}-1}$ and $d_{\mu}\in (c_{\mu-1}^*)^{\op}$, Lemma~\ref{lemma:cdnullhomotopic} (applied to $(c,d):=(c_{\mu}^*,c_{\mu-1}^*)$ and $(x_0,\dots,x_k):=X_{\mu}$) provides a gallery $Y_{\mu}=(d_{\mu}^*=y_{\mu 0},\dots,y_{\mu l_{\mu}}=d_{\mu})$ contained in $(c_{\mu-1}^*)^{\op}$ such that the gallery
$$\overline{X}_{\mu}((c_{\mu}^*,d_{\mu}^*),(c_{\mu-1}^*,d_{\mu}^*))\overline{Y}_{\mu}((c_{\mu-1}^*,d_{\mu}),(c_{\mu}^*,d_{\mu}))$$
of $\Opp(\CCC)$ is null-homotopic.

For $\mu\in\{0,\dots,k^*-1\}$, consider the gallery $G_{\mu}=((c_{\mu}^*,d_{\mu}^*=c^-_{\lambda_{\mu}}),\dots,(c_{\mu}^*,c^-_{\lambda_{\mu+1}-1}=d_{\mu+1}^*))$ in $\Opp(\CCC)$, and set $G_{k^*}=((c_{k^*}^*,d_{k^*}^*=c^-_{\lambda_{k^*}}),\dots,(c_{k^*}^*,c^-_{k}))$. Then $G$ can be decomposed as
$$G=G_0((c_0^*,d_1^*),(c_1^*,d_1^*))G_1((c_1^*,d_2^*),(c_2^*,d_2^*))G_2\dots G_{k^*-1}((c_{k^*-1}^*,d_{k^*}^*),(c_{k^*}^*,d_{k^*}^*))G_{k^*}.$$
Since for each $\mu\in\{1,\dots,k^*\}$, the galleries $((c_{\mu-1}^*,d_{\mu}^*),(c_{\mu}^*,d_{\mu}^*))$ and $\overline{Y}_{\mu}((c_{\mu-1}^*,d_{\mu}),(c_{\mu}^*,d_{\mu}))\overline{X}_{\mu}$ are homotopic by the previous paragraph, we deduce that $G$ is homotopic to 
$$G'=G_0\overline{Y}_{1}((c_{0}^*,d_{1}),(c_{1}^*,d_{1}))\overline{X}_{1}G_1\overline{Y}_{2}\dots \overline{X}_{k^*-1}G_{k^*-1}\overline{Y}_{k^*}((c_{k^*-1}^*,d_{k^*}),(c_{k^*}^*,d_{k^*}))\overline{X}_{k^*}G_{k^*}.$$
On the other hand, Proposition~\ref{prop:null-homotopicccc} implies that $G_0\overline{Y}_1$ is homotopic to $\overline{Z}_0$, that $\overline{X}_{\mu}G_{\mu}\overline{Y}_{\mu+1}$ is homotopic to $\overline{Z}_{\mu}$ for all $\mu\in\{1,\dots,k^*-1\}$, and that $\overline{X}_{k^*}G_{k^*}$ is homotopic to $\overline{Z}_{k^*}$. In particular, $G'$ is homotopic to 
$$G''=\overline{Z}_0((c_{0}^*,d_{1}),(c_{1}^*,d_{1}))\overline{Z}_1\dots \overline{Z}_{k^*-1}((c_{k^*-1}^*,d_{k^*}),(c_{k^*}^*,d_{k^*}))\overline{Z}_{k^*}.$$
But $G''$ is a closed gallery in $R_J((c^+,c^-))$ and is therefore null-homotopic, as desired.
\end{proof}

\begin{lemma}\label{lemma:GXsimple}
Let $G$ be a gallery in $\Opp(\CCC)$ and let $X$ be a simple gallery in $\CCC_+$ such that $G_+^*$ and $X$ are elementary homotopic. Then there exist a gallery $H$ in $\Opp(\CCC)$ such that $H$ is homotopic to $G$ and such that $H^*_+=X$.
\end{lemma}
\begin{proof}
Let $\lambda(G_+)=(\lambda_1,\dots,\lambda_k)$, let $G^*_+=(c_0^*,\dots,c^*_k)$, let $X=(c_0',\dots,c_l')$ and let $J,\mu,\nu,\mu',\nu'$ be as in the definition of an elementary homotopy from $G^*_+$ to $X$: we have
$$c_0^*=c_0', \ c_1^*=c_1',\ \dots, \ c_{\mu}^*=c'_{\mu'}\quad\textrm{and}\quad c_{\nu}^*=c'_{\nu'}, \ c_{\nu+1}^*=c'_{\nu'+1}, \ \dots, \ c_k^*=c'_l,$$
while the $J$-galleries $(c_{\mu}^*,\dots,c_{\nu}^*)$ and $(c'_{\mu'},\dots,c'_{\nu'})$ are homotopic. 

Write $G=G_1(c_{\lambda_{\mu}},\dots,c_{\lambda_{\nu}})G_2$, as well as $c_{\lambda_{\mu}}=(c_{\mu}^*,a)$ and $c_{\lambda_{\nu}}=(c_{\nu}^*,b)$. 
Put $l'=\nu'-\mu'$. Since $c'_{\mu'}=c_{\mu}^* \op a$, a repeated use of (TCS2) allows to inductively construct a gallery $(a=a_0,\dots,a_{l'})$ in $\CCC_-$ such that 
$((c'_{\mu'},a_0),(c'_{\mu'+1},a_1),\dots,(c'_{\nu'},a_{l'}))$ is a gallery in $\Opp(\CCC)$. Since $c'_{\nu'}=c_{\nu}^* \op b$, we can also choose by (TCS3) a gallery $(a_{l'}=b_0,\dots,b_{l''}=b)$ contained in $(c'_{\nu'})^{\op}$. 

Set $$H=G_1((c'_{\mu'},a_0),(c'_{\mu'+1},a_1),\dots,(c'_{\nu'},a_{l'})=(c'_{\nu'},b_0),(c'_{\nu'},b_1),\dots,(c'_{\nu'},b_{l''}))G_2.$$
Then $G$ is homotopic to $H$ by Proposition~\ref{prop:mainwithJ}, since $(c'_{\mu'},\dots,c'_{\nu'})$ is contained in a $J$-residue. Moreover, since $(c'_{\mu'},c'_{\mu'+1},\dots,c'_{\nu'})$ is a simple gallery, we have $$H_+^*=(c_0^*,\dots,c_{\mu}^*)(c_{\mu^*}=c'_{\mu'},c'_{\mu'+1},\dots,c'_{\nu'}=c_{\nu}^*)(c_{\nu}^*,\dots,c_k^*)=X,$$ as desired.
\end{proof}

If a gallery $G$ is (resp. elementary) homotopic to a gallery $H$, we write $G\sim H$ (resp. $G\to H$).

\begin{theorem}
Let $\CCC$ be a simply connected twin chamber system. Then $\Opp(\CCC)$ is simply connected.
\end{theorem}
\begin{proof}
Let $G$ be a closed gallery in $\Opp(\CCC)$. Then $G_+$ is a closed gallery in $\CCC_+$. Since $\CCC_+$ is simply connected by assumption, $G^*_+$ can be transformed to a length zero gallery $(c^+)$ by a sequence of elementary homotopies, say $G^*_+ \to X_1\to X_2 \to\dots\to X_r=(c^+)$ with each $X_{\mu}$ a simple gallery. Applying Lemma~\ref{lemma:GXsimple} inductively, we find galleries $H_{\mu}$ in $\Opp(\CCC)$ such that $G\sim H_1\sim H_2\sim\dots\sim H_r$ and $H^*_{\mu+}=X_{\mu}$ for all $\mu=1,\dots,r$. As $H_r$ is null-homotopic by Proposition~\ref{prop:null-homotopicccc}, the claim follows.
\end{proof}


\section{The opposition datum of \texorpdfstring{$G^{\min}_R$}{GminR}}\label{section:TCSofGminR}
Let $R$ be a local ring, with residue field $k$. Assume that $A$ is $2$-spherical and that $R$ satisfies (co). The purpose of this section is to associate a twin chamber system to $G^{\min}_R$, to which we can apply the results of \S\ref{section:TCSdef}.

\begin{definition}
 We define the two chamber systems $\CCC_{\pm}=\CCC_{\pm}(G^{\min}_R)$ over $I$ as follows. The set of chambers of $\CCC_{\pm}$ is $G^{\min}_R/\oB^{\pm}_R$. Set $c_0^{\pm}:=\oB^{\pm}_R\in\CCC_{\pm}$. For $g,h\in G^{\min}_R$, the chambers $gc_0^{\pm}$ and $hc_0^{\pm}$ are {\bf $i$-adjacent} for some $i\in I$ (denoted $gc_0^{\pm} \sim_i hc_0^{\pm}$) if $h\inv g\in \oP^{\pm}_{iR}$. Note that $\oP^{\pm}_{iR}=\oG_{iR}\oB^{\pm}_R=\oB^{\pm}_R\oG_{iR}$ by Lemma~\ref{lemma:BGiisGiB}.

Two chambers $gc_0^{+}$ and $hc_0^{-}$ with $g,h\in G^{\min}_R$ are {\bf opposite} (denoted $gc_0^+ \op_R hc_0^-$) if $h\inv g\in \oB^{-}_R\oB^{+}_R$. 
We denote by $\CCC=\CCC(G^{\min}_R)$ the {\bf opposition datum of $G^{\min}_R$}, consisting of the pair $(\CCC_+,(\sim_i)_{i\in I})$, $(\CCC_-,(\sim_i)_{i\in I})$ of chamber systems together with the opposition relation $\op=\op_R$. Note that $G^{\min}_R$ acts transitively on $\CCC$.
\end{definition}

\begin{lemma}
The opposition datum $\CCC$ satisfies the axiom (TCS1).
\end{lemma}
\begin{proof}
WLOG we may assume that $\epsilon=+$ (the proof is similar for $\epsilon=-$), that $c=c_0^+$ and that $x=c_0^-$ (because $G^{\min}_R$ is transitive on $\Opp(\CCC)$). 

Since $d \in (c_0^-)^{\op_R}=\oU^{-}_Rc_0^+$ and $d\sim_i c_0^+$ (that is, $d\in \oG_{iR}c_0^+$), Lemma~\ref{lemma:GicapB-B+} implies that $d=x_{-i}(r)c_0^+$ for some $r\in R$. Similarly, $y=x_j(s)c_0^-$ for some $s\in R$. If $i\neq j$, then $d=x_{-i}(r)c_0^+ \op_R x_j(s)c_0^-=y$ because $x_j(-s)x_{-i}(r)c_0^+=x_{-i}(r)x_j(-s)c_0^+=x_{-i}(r)c_0^+ \op_R c_0^-$.
\end{proof}

\begin{lemma}
The opposition datum $\CCC$ satisfies the axiom (TCS2).
\end{lemma}
\begin{proof}
As in the previous lemma, we may assume that $\epsilon=+$, that $c=c_0^+$ and that $x=c_0^-$. Since $R$ is local, it has stable rank $1$, and hence $\SL_2(R)=U_2^+(R)U_2^-(R)B_2^+(R)$. There thus exist $r,s\in R$ such that $d=x_i(r)x_{-i}(s)c_0^+$, and we can choose $y=x_i(r)c_0^-$.
\end{proof}

\begin{lemma}
The opposition datum $\CCC$ satisfies the axiom (TCS3).
\end{lemma}
\begin{proof}
To ensure that $\CCC$ satisfies (TCS3), it suffices to check that $\oU^{\pm}_R=\langle U_{\pm\alpha_i}(R) \ | \ i\in I\rangle$ (and similarly for $A$ replaced by $A_J=(a_{ij})_{i,j\in J}$ for any $J\subseteq I$ with $|J|=2$). This thus follows from Proposition~\ref{prop:simplegenerationU+}.
\end{proof}

\begin{lemma}
The opposition datum $\CCC$ satisfies the axiom (TCS4).
\end{lemma}
\begin{proof}
WLOG we may assume that $\epsilon=+$ (the proof is similar for $\epsilon=-$), and that $c=c_0^+$ (because $G^{\min}_R$ is transitive on $\Opp(\CCC)$). 

Recall from Remark~\ref{remark:Bruhat_and_Levidec_over_fields}(1) that the pairs $(B^{\pm}_k,N_k)$ form a twin BN-pair for $G^{\min}_k=G_k$. In particular, the chamber systems $\Delta^{\pm}_k:=(\CCC_{\pm}(G_k),(\sim_{i})_{i\in I})$ form a twin building (see \cite[\S6.3.3]{BrownAbr}), and we denote by $\delta^*_k$ the associated codistance function on $(\Delta^+_k,\Delta^-_k)$, as in \cite[Definition~5.133]{BrownAbr} (in particular, given two chambers $c_+$ of $\Delta^+_k$ and $c_-$ of $\Delta^-_k$, we have $\delta^*_k(c_+,c_-)=1_{\WW}$ if and only if $c_+ \op_k c_-$). Let also $\delta^{\pm}_k\co \Ch(\Delta^{\pm}_k)\times\Ch(\Delta^{\pm}_k)\to\WW$ denote the Weyl distance on $\Delta^{\pm}_k$ as in \cite[Definition~5.1.1]{BrownAbr}, where $\Ch(\Delta^{\pm}_k)=G_k/B^{\pm}_k$ is the set of chambers of $\Delta^{\pm}_k$. 

Let $$\opi_R\co G_R^{\min}\to G^{\min}_k=G_k$$ be the canonical map, so that $\pi_R=\opi_R\circ\varphi_R$, and consider the chamber maps
$$\wpi_R\co \CCC_{\pm}(G^{\min}_R)\to \Delta^{\pm}_k:g\oB^{\pm}_R\mapsto \opi_R(g)B^{\pm}_k.$$ Write $\overline{c}_0^{\pm}=\wpi_R(c_0^{\pm})=B^{\pm}_k$ for the fundamental chamber of $\Delta^{\pm}_k$. 

Let $\Sigma^+_R:=\{\widetilde{w}c_0^+ \ | \ w\in\WW\}\subseteq \CCC_{+}(G^{\min}_R)$, so that $\wpi_R|_{\Sigma^+_R}\co\Sigma^+_R\to\Sigma^+_k$ is an isomorphism from $\Sigma^+_R$ onto the fundamental apartment $\Sigma^+_k=\{\widetilde{w}\overline{c}_0^+ \ | \ w\in\WW\}$ of $\Delta^+_k$. For each $x\in \Ch(\Delta^{-}_k)$, let $\sigma_k(x)$ denote the unique chamber of $\Sigma^+_k$ such that $\delta^*_k(\overline{c}_0^+,x)=\delta^+_k(\overline{c}_0^+,\sigma_k(x))$. Then $\sigma_k\co \Ch(\Delta^{-}_k)\to \Sigma^+_k$ is a chamber map (see \cite[Lemma~5.139(1)]{BrownAbr}) such that $x \op_k\sigma_k(x)$ for all $x\in\Ch(\Delta^{-}_k)$ (see \cite[Corollary~5.141(1)]{BrownAbr}).

Define the chamber map $\omega_c\co \CCC_-(G^{\min}_R)\to \Sigma^+_R\subseteq\CCC_+(G^{\min}_R)$ by setting $$\omega_c(x):=(\wpi_R|_{\Sigma^+_R})\inv\sigma_k(\wpi_R(x))\quad\textrm{for all $x\in \CCC_-(G^{\min}_R)$.}$$
Let now $x\in\CCC_-(G^{\min}_R)$, say $x=gc_0^-$ for some $g\in G^{\min}_R$, so that $\wpi_R(x)=\opi_R(g)\overline{c}_0^{-}$. Write $\sigma_k(\wpi_R(x))=\widetilde{w}\overline{c}_0^{+}$ for some $w\in\WW$. Since $\sigma_k(\wpi_R(x)) \op_k \wpi_R(x)$, we have $\opi_R(g)\inv \widetilde{w}\overline{c}_0^{+} \op_k \overline{c}_0^{-}$, and hence $\opi_R(g\inv \widetilde{w})=\opi_R(g)\inv \widetilde{w}\in B^-_kB^+_k$. By Corollary~\ref{corollary:oppositionRandk}, this implies that $g\inv \widetilde{w}\in \oB^{-}_R\oB^{+}_R$, that is, $g\inv \widetilde{w} c_0^+ \op_R c_0^-$. Hence $\omega_c(x)=\widetilde{w}c_0^+ \op_R gc_0^-=x$.

Finally, if $x\in c^{\op_R}$, then $x=gc_0^-$ for some $g\in \oB^{+}_R$, and hence $\wpi_R(x)=\opi_R(g)\overline{c}_0^- \op_k \overline{c}_0^+$. In particular, $\omega_c(x)=(\wpi_R|_{\Sigma^+_R})\inv( \overline{c}_0^+)=c$.
\end{proof}

\begin{corollary}\label{corollary:TCS}
Assume that $A$ is $2$-spherical and that $R$ is a local ring satisfying (co). Then $\CCC(G^{\min}_R)$ is a twin chamber system.
\end{corollary}

\begin{lemma}\label{lemma:CCCsimplyconnectedKM}
Assume that $R$ is a Bezout domain, with field of fractions $\KK$. Then the map $$\CCC_{\pm}(G^{\min}_R)\to \CCC_{\pm}(G_{\KK}):g\oB^{\pm}_R\to gB^{\pm}_{\KK}$$ is an isomorphism of chamber systems. In particular, $\CCC$ is simply connected.
\end{lemma}
\begin{proof}
By Lemma~\ref{lemma:sl2kK_Steinberg}, we have $\SL_2(\KK)=\SL_2(R) B^+_2(\KK)$. In particular, for all $i_0,\dots,i_d\in I$, we have $\oG_{i_0\KK}\dots\oG_{i_d\KK}B^{\pm}_{\KK}\subseteq\oG_{i_0R}\dots\oG_{i_dR}B^{\pm}_{\KK}$, as follows from an induction on $d$ using that
$$\oG_{i_0\KK}\oG_{i_1\KK}\dots\oG_{i_d\KK}B^{\pm}_{\KK}\subseteq\oG_{i_0R}B^{\pm}_{\KK}\oG_{i_1\KK}\dots\oG_{i_d\KK}B^{\pm}_{\KK}\subseteq \oG_{i_0R}\oG_{i_1\KK}\dots\oG_{i_d\KK}B^{\pm}_{\KK},$$
where the last inclusion follows from the fact that $B^{\pm}_{\KK}\oG_{i\KK}=\oG_{i\KK}B^{\pm}_{\KK}$ for all $i\in I$ (see Lemma~\ref{lemma:BGiisGiB}). This shows that $$G_{\KK}=\G_A^{\min}(R)B^{\pm}_{\KK}.$$
Since $B^{\pm}_{\KK}\cap G_A^{\min}(R)=\oB^{\pm}_R$ and $P^{\pm}_{i\KK}\cap G_A^{\min}(R)=\oP_{iR}^{\pm}$ for all $i\in I$ by Theorem~\ref{thm:UbarRisGminRcapUK}, the map $\CCC_{\pm}(G^{\min}_R)\to \CCC_{\pm}(G_{\KK}):g\oB^{\pm}_R\to gB^{\pm}_{\KK}$ is an isomorphism of chamber systems. Since $\CCC_{\pm}(G_{\KK})$ are buildings and hence simply connected (see \cite[Theorem~4.3]{Ron89}), the claim follows.
\end{proof}

Recall that a local ring that is a Bezout domain is precisely a \textbf{valuation ring}, namely, a domain $R$ with field of fractions $\KK$ such that for any nonzero $x\in\KK$, at least one of $x$ or $x\inv$ belongs to $R$. Corollary~\ref{corollary:TCS} and Lemma~\ref{lemma:CCCsimplyconnectedKM} sum up to the following theorem.
\begin{theorem}\label{thm:KMlocalOppSC}
Assume that $A$ is $2$-spherical and that $R$ is a valuation ring satisfying (co). Then $\CCC(G^{\min}_R)$ is a simply connected twin chamber system, and $G^{\min}_R$ acts transitively on $\CCC(G^{\min}_R)$.
\end{theorem}

We can now prove our main theorem.

\begin{theorem}\label{thm:main_amalgamKM}
Assume that $A$ is $2$-spherical and that $R$ is a valuation ring satisfying (co).
Then $G^{\min}_R$ is the amalgamated product of the system of subgroups $\{T_R\oG_{JR}=T_RG_{JR} \ | \ |J|\leq 2\}$. In particular, the morphism $\varphi_R\co G_R\to G_R^{\min}$ is an isomorphism, which restricts to isomorphisms $U^{\pm}_{R}\to U^{\pm}_{\KK}\cap G^{\min}_R$, where $\KK$ is the field of fractions of $R$.
\end{theorem}
\begin{proof}
Note that the restriction of $\varphi_R$ to the subgroups $T_RG_{JR}$ is injective by Theorem~\ref{theorem:spherical_injectivity}, whence the equality $T_R\oG_{JR}=T_RG_{JR}$. By Theorem~\ref{thm:KMlocalOppSC}, $\CCC(G^{\min}_R)$ is a simply connected twin chamber system and $G^{\min}_R$ acts transitively on $\CCC(G^{\min}_R)$. Moreover, if $J\subseteq I$ with $|J|\leq 2$, then the stabiliser in $G^{\min}_R$ of the $J$-residues of $c_+$ and $c_-$ is $\oP_{JR}^{+}\cap \oP_{JR}^-$, and therefore coincides with $T_R\oG_{JR}$ by Proposition~\ref{prop:intersection_parabolic_opposite_signs} (recall that domains satisfy (\ref{Bir})). The first statement thus follows from Corollary~\ref{corollary:sctwincs_implies_CTA}.

Denote by $\mathrm{CT}_{\DDD}(R)$ the amalgamated product of the subgroups $\{T_RG_{JR} \ | \ |J|\leq 2\}$, so that we have an isomorphism $\mathrm{CT}_{\DDD}(R)\to G^{\min}_R$. Since the relations defining $\mathrm{CT}_{\DDD}(R)$ are satisfied in $G_R$, we then have canonical injective morphisms
\begin{equation}
\mathrm{CT}_{\DDD}(R)\stackrel{\sim}{\to} \G_{\DDD}(R)\stackrel{\sim}{\to}\G^{\min}_{\DDD}(R)\hookrightarrow \G^{\min}_{\DDD}(\KK)=\G_{\DDD}(\KK).
\end{equation}
The fact that the isomorphism $\varphi_R\co G_R\to G_R^{\min}$ restricts to  isomorphisms $U^{\pm}_{R}\to U^{\pm}_{\KK}\cap G^{\min}_R$ follows from Theorem~\ref{thm:UbarRisGminRcapUK}(3).
\end{proof}

\begin{remark}
Let $A=(a_{ij})_{i,j\in I}$ and $R$ be as in Theorem~\ref{thm:main_amalgamKM}, and assume that $\DDD=\DDD_A^{\mathrm{sc}}$ (see \S\ref{subsection:KMRDandtori}), so that $T_R=\langle r^{h_i} \ | \ r\in R^{\times}, \ i\in I\rangle\cong (R^{\times})^{|I|}$. For each $J\subseteq I$ with $|J|\leq 2$, set $$T_{JR}:=\langle  r^{h_i} \ | \ r\in R^{\times}, \ i\in J\rangle\quad\textrm{and}\quad T^J_R:=\langle  r^{h_i} \ | \ r\in R^{\times}, \ i\in I\setminus J\rangle,$$ so that $T_R=T_{JR}\times T^J_R$. Note that $G_{JR}\cap T_R=T_{JR}$ (this follows for instance from Remark~\ref{remark:Bruhat_and_Levidec_over_fields}(2) applied to $G_{J\KK}$ and $G_{\KK}$). Hence $T_RG_{JR}=T^{J}_R\ltimes G_{JR}$ is the quotient of the free product $T^{J}_R* G_{JR}$ by the relations $$r^{h_i}x_{\pm\alpha_j}(a)(r^{h_i})\inv = x_{\pm\alpha_j}(ar^{\pm a_{ij}})\quad\textrm{for all $i\in I\setminus J$, \ $j\in J$, \ $r\in R^{\times}$ and $a\in R$.}$$
Since these relations are satisfied in $G_{\{i,j\}R}$, and similarly the defining relations 
$$r^{h_i}s^{h_i}=(rs)^{h_i}\quad\textrm{and}\quad r^{h_i}s^{h_j}=s^{h_j}r^{h_i}\quad\textrm{for all $i,j\in I\setminus J$ and $r,s\in R^{\times}$}$$
of $T^{J}_R$ are satisfied in $G_{\{i,j\}R}$, it readily follows that $\mathrm{CT}_{A}(R)=\mathrm{CT}_{\DDD}(R)$ is the amalgamated product of the system of subgroups $\{G_{JR} \ | \ |J|\leq 2\}$.
\end{remark}

\begin{remark}\label{remark:proofPropB}
To see that Theorem~\ref{thm:main_amalgamKM} implies Corollary~\ref{corintro:Laurentpolynuniversal}, consider a valuation ring $R$ satisfying (co) and with field of fractions $\KK$, and let $\bar{A}=(a_{ij})_{i,j\in \bar{I}}$ be a Cartan matrix. Write $\bar{I}=\{1,\dots,\ell\}$ and set $\bar{\DDD}:=\DDD_{\bar{A}}^{\sico}$.  
Let $\bar{\Phi}$ be the root system of $\bar{A}$, with simple roots $\alpha_1,\dots,\alpha_{\ell}$, and let $\theta$ denote the highest root of $\bar{\Phi}$. Assume that $\bar{\Phi}$ is irreducible and not of type $A_1$. Recall that $\G_{\bar{A}}(\KK[t,t\inv])=\CDem_{\bar{\Phi}}(\KK[t,t\inv])$ by \cite{Mor82}, where $\G_{\bar{A}}=\G_{\bar{\DDD}}$. We have to show that the canonical map $$\bar{\psi}_R\co\G_{\bar{A}}(R[t,t\inv])\to \G_{\bar{A}}(\KK[t,t\inv])=\CDem_{\bar{\Phi}}(\KK[t,t\inv])$$ is injective. 

Recall from Remark~\ref{remark:presentation_Chevalley_groups} that $\G_{\bar{A}}(R[t,t\inv])$ has generators $\{\bar{x}_{\alpha}(P) \ | \ P\in R[t,t\inv]\}$ satisfying the following relations, for all $\alpha,\beta\in\bar{\Phi}$, $m,n\in\ZZ$, $a,b\in R$ and $r,s\in R^{\times}$, where we set 
$$n_{\alpha}(rt^m):=\bar{x}_{\alpha}(rt^m)\bar{x}_{-\alpha}(r\inv t^{-m})\bar{x}_{\alpha}(rt^m)\quad\textrm{and}\quad h_{\alpha}(rt^m):=n_{\alpha}(1)\inv n_{\alpha}(r\inv t^{-m}):$$
\begin{enumerate}
\item[($\bar{\mathrm{U}}$)]
$\bar{x}_{\alpha}(at^m)\bar{x}_{\alpha}(bt^n)=\bar{x}_{\alpha}(at^m+bt^n)$,
\item[($\bar{\mathrm{C}}$)]
For $\beta\neq\pm\alpha$: 
$$[\bar{x}_{\alpha}(at^m),\bar{x}_{\beta}(bt^n)]=\prod_{\stackrel{\gamma\in ]\alpha,\beta[_{\NN}}{\gamma=i\alpha+j\beta}}\bar{x}_{\gamma}(\bar{C}^{\alpha\beta}_{ij}a^ib^jt^{im+jn}),$$
\item[($\bar{\mathrm{T}}$)]
$h_{\alpha}(rt^m)h_{\alpha}(st^n)=h_{\alpha}(rst^{m+n})$.
\end{enumerate}
Set $I:=\{0,1,\dots,\ell\}$, and let $A=(a_{ij})_{i,j\in I}$ be the extended matrix of $A$ (see  \cite[\S5.3]{KMGbook}). Since $\bar{A}$ is not of type $A_1$, the GCM $A$ is $2$-spherical. Let $\DDD$ be the Kac--Moody root datum associated to $A$ defined in \cite[p.160]{KMGbook} (in particular, $T_R=\langle r^{h_i} \ | \ r\in R^{\times}, \ i\in\bar{I}\rangle$ is the same torus for both $\DDD$ and $\bar{\DDD}$), so that Theorem~\ref{thm:main_amalgamKM} yields the injectivity of the canonical morphism
\begin{equation*}
\psi_R\co\mathrm{CT}_{\DDD}(R)\to\G_{\DDD}(\KK)=\CDem_{\bar{\Phi}}(\KK[t,t\inv]).
\end{equation*}
It is given by the following assignment, for all $i\in\bar{I}$ and $a\in R$ (see \cite[\S7.6]{KMGbook}):
\begin{equation}
x_{\pm\alpha_i}(a)\mapsto\bar{\psi}_R(\bar{x}_{\pm\alpha_i}(a)), \quad x_{\alpha_0}(a)\mapsto \bar{\psi}_R(\bar{x}_{-\theta}(at))\quad\textrm{and}\quad x_{-\alpha_0}(a)\mapsto \bar{\psi}_R(\bar{x}_{\theta}(at\inv)).
\end{equation}
To show that $\bar{\psi}_R$ is injective, it thus suffices to check that the assignment
\begin{equation}
x_{\pm\alpha_i}(a)\mapsto \bar{x}_{\pm\alpha_i}(a), \quad x_{\alpha_0}(a)\mapsto \bar{x}_{-\theta}(at), \quad x_{-\alpha_0}(a)\mapsto \bar{x}_{\theta}(at\inv)\quad\textrm{for $i\in\bar{I}$ and $a\in R$}
\end{equation}
defines, for each $J\subseteq I$ with $|J|\leq 2$, a morphism $T_RG_{JR}\to \G_{\bar{A}}(R[t,t\inv])$, and hence also a surjective morphism $\mathrm{CT}_{\DDD}(R)\to \G_{\bar{A}}(R[t,t\inv])$ whose composition with $\bar{\psi}_R$ coincides with $\psi_R$ (the surjectivity follows for instance from \cite[Lemma~7.89]{KMGbook}). Noting that under this assignment, 
\begin{equation}\widetilde{s}_0(r):=x_{\alpha_0}(r)x_{-\alpha_0}(r\inv)x_{\alpha_0}(r)\mapsto \bar{x}_{-\theta}(rt)\bar{x}_{\theta}(r\inv t\inv)\bar{x}_{-\theta}(rt)=n_{-\theta}(rt)
\end{equation}
and 
\begin{equation}r^{h_0}:=\widetilde{s}_0(1)\inv\widetilde{s}_0(r\inv)\mapsto n_{-\theta}(t)\inv n_{-\theta}(r\inv t)=h_{-\theta}(t\inv)\inv h_{-\theta}(rt\inv)=h_{-\theta}(r)
\end{equation}
for all $r\in R^{\times}$, it is now straightforward to check that the image in $\G_{\bar{A}}(R[t,t\inv])$ of the defining relations of $T_RG_{JR}$ (see for instance Definition~\ref{definition:constructiveTitsfunctor} for the group $\G_{\DDD(J)}(R)=T_RG_{JR}$ as in the proof of Lemma~\ref{lemma:GJRcappiRinvB}) indeed follow from the relations ($\bar{\mathrm{U}}$), ($\bar{\mathrm{C}}$) and ($\bar{\mathrm{T}}$) (note that the integral constants in the commutator relations in $T_RG_{JR}$ and $\G_{\bar{A}}(R[t,t\inv])$ are necessarily the same, since both coincide with the corresponding constants in $\CDem_{\bar{\Phi}}(\KK[t,t\inv])$).
\end{remark}

\bibliographystyle{amsalpha} 
\bibliography{injectivityKMbib}
\end{document}